\pdfoutput=1
\RequirePackage{ifpdf}
\ifpdf 
\documentclass[pdftex]{sigma}
\else
\documentclass{sigma}
\fi

\usepackage{stmaryrd}

\usepackage{braket}
\usepackage{accents}
\usepackage{url}
\usepackage[all]{xy}

\numberwithin{equation}{section}

\newtheorem{thm}{Theorem}[section]
\newtheorem{prop}[thm]{Proposition}
\newtheorem{lem}[thm]{Lemma}
\newtheorem{cor}[thm]{Corollary}

\theoremstyle{definition}
\newtheorem{dfn}[thm]{Definition}
\newtheorem{rmk}[thm]{Remark}

\newcommand{\Zz}{\mathbb{Z}}
\newcommand{\Cz}{\mathbb{C}}
\newcommand{\hyper}[4]{\left(\begin{matrix} #1 \\ #2 \end{matrix}; #3, #4\right)}
\newcommand{\sn}{\mathrm{sn}\,}
\newcommand{\cn}{\mathrm{cn}\,}
\newcommand{\dn}{\mathrm{dn}\,}

\begin{document}
\allowdisplaybreaks

\newcommand{\arXivNumber}{2206.15137}

\renewcommand{\PaperNumber}{014}

\FirstPageHeading

\ShortArticleName{A Generalization of Zwegers' $\mu$-Function}

\ArticleName{A Generalization of Zwegers' $\boldsymbol{\mu}$-Function According\\ to the $\boldsymbol{q}$-Hermite--Weber Difference Equation}

\Author{Genki SHIBUKAWA and Satoshi TSUCHIMI}

\AuthorNameForHeading{G.~Shibukawa and S.~Tsuchimi}

\Address{Department of Mathematics, Kobe University, Rokko, 657--8501, Japan}
\Email{\href{mailto:g-shibukawa@math.kobe-u.ac.jp}{g-shibukawa@math.kobe-u.ac.jp}, \href{mailto:183s014s@stu.kobe-u.ac.jp}{183s014s@stu.kobe-u.ac.jp}}

\ArticleDates{Received July 02, 2022, in final form February 25, 2023; Published online March 23, 2023}

\Abstract{We introduce a one parameter deformation of the Zwegers' $\mu$-function as the image of $q$-Borel and $q$-Laplace transformations of a fundamental solution for the $q$-Hermite--Weber equation. We further give some formulas for our generalized $\mu$-function, for example, forward and backward shift, translation, symmetry, a difference equation for the new parameter, and bilateral $q$-hypergeometric expressions. From one point of view, the continuous $q$-Hermite polynomials are some special cases of our $\mu$-function, and the Zwegers' $\mu$-function is regarded as a continuous $q$-Hermite polynomial of ``$-1$ degree''.}

\Keywords{Appell--Lerch series; $q$-Boerl transformation; $q$-Laplace transformation; $q$-hyper\-geometric series; continuous $q$-Hermite polynomial; mock theta functions}

\Classification{33D15; 39A13; 30D05; 11F50; 33D70}
\section{Introduction}
Throughout this paper, let $\Zz$ denote the set of integers and let $\Cz$ denote the set of complex numbers, ${\rm i}:=\sqrt{-1}$ be the imaginary unit, $\tau\in\Cz$ be a complex number $\mathrm{Im}(\tau)>0$, and $q:={\rm e}^{2\pi{\rm i}\tau}$.
We define the $q$-shifted factorials and Jacobi theta functions as follows:
\begin{gather*}
(x)_{\infty }
 =
 (x;q)_{\infty }
 :=
 \prod_{j=0}^{\infty }\big(1-xq^{j}\big), \qquad
 (x)_{n }=(x;q)_{n}:=\frac{(x;q)_{\infty }}{(q^{n}x;q)_{\infty }}, \qquad n \in \mathbb{Z}, \\
\vartheta_{11}(u,\tau)
 =
 \vartheta_{11}(u)
 :=
 \sum_{n\in\Zz}{\rm e}^{2\pi{\rm i}(n+\frac{1}{2})(u+\frac{1}{2})+\pi{\rm i}(n+\frac{1}{2})^2\tau}
 =-{\rm i}q^{\frac{1}{8}}{\rm e}^{-\pi{\rm i}u}\big(q,{\rm e}^{2\pi{\rm i}u}, q{\rm e}^{-2\pi{\rm i}u};q\big)_{\infty }, \\
\theta_q(x)
 :=(q,-x,-q/x)_\infty
 =\sum_{n \in \mathbb{Z}} x^nq^\frac{n(n-1)}{2},
\end{gather*}
and for appropriate complex numbers $a_1,\dots, a_r,b_1,\dots,b_s,x$, we define the $q$-hypergeometric series as follows:
\begin{gather*}
{}_r\phi_{s}\left( \begin{matrix} a_1,\dots,a_r \\ b_1,\dots,b_s \end{matrix};q,x\right)
 :=
 \sum_{n=0}^{\infty }\frac{(a_1,\dots, a_r)_n}{(b_1,\dots,b_s,q)_n}\left((-1)^{n }q^\frac{n(n-1)}{2}\right)^{s-r+1}x^{n }, \\
{}_r\psi_{s}\left( \begin{matrix} a_1,\dots, a_r\\b_1,\dots, b_s\end{matrix};q,x\right)
 :=
 \sum_{n \in \mathbb{Z}}\frac{(a_1,\dots, a_r)_n}{(b_1,\dots, b_s)_n}\left((-1)^{n }q^\frac{n(n-1)}{2}\right)^{s-r}x^{n },
\end{gather*}
where
\begin{gather*}
(a_1,\dots,a_r)_{n }
 =
 (a_1,\dots,a_r;q)_{n }
 :=(a_1;q)_{n}\cdots (a_r;q)_{n}, \qquad n \in \mathbb{Z}\cup \{\infty \}.
\end{gather*}

Mock theta functions first appeared in Ramanujan's last letter to Hardy in 1920.
In this letter, Ramanujan told Hardy that he had discovered a new class of functions which he called mock theta functions.
Mock theta functions are functions that have asymptotic behavior similar to ``theta functions'' (i.e., modular forms) at roots of unity but are not ``theta functions'', and Ramanujan gave 17 examples of mock theta functions.
Some typical examples are as follows:
\begin{gather*}
f_0(q)=\sum_{n=0}^{\infty } \frac{q^{n^2}}{(-q;q)_n}, \qquad
\phi(q)=\sum_{n=0}^{\infty } \frac{q^{n^2}}{\big({-}q^2;q^2\big)_n}, \qquad
\psi(q)=\sum_{n=1}^{\infty } \frac{q^{n^2}}{\big(q;q^2\big)_n}.
\end{gather*}
Later, Andrews and Hickerson gave a detailed definition of the mock theta function~\cite{AH}. For more background on mock theta functions, see, for example,~\cite{AB,GM}.

Ramanujan's mock theta functions are represented as linear combinations of specializations of the universal mock theta function:
\[
g_3(x;q):=\sum_{n=1}^{\infty }\frac{q^{n(n-1)}}{(x)_{n }\big(x^{-1}q\big)_{n }}
\]
and some $q$-infinite products, which have come to be known as the mock theta conjectures.
For example \cite[Appendix~A]{BFOR},
\begin{gather*}
f_0(q)
 =
 -2q^{2}g_{3}\big(q^{2};q^{10}\big)+\frac{\big(q^{5};q^{5}\big)_{\infty }\big(q^{5};q^{10}\big)_{\infty }}{\big(q;q^{5}\big)_{\infty }\big(q^{4};q^{5}\big)_{\infty }}.
\end{gather*}
Such identities were proved by Hickerson \cite{H}.
More fundamentally, the universal mock theta function $g_{3}(x;q)$ can be written in the form~\cite[Theorem~3.1]{K}
\begin{gather*}
q^{-\frac{1}{24}}x^\frac{3}{2}g_3(x;q)
 =
 \frac{q^{\frac{1}{3}}\big(q^{3};q^{3}\big)_{\infty }^{3}}{(q;q)_{\infty }\vartheta_{11}(3u,3\tau)} +q^{-\frac{1}{6}}x\mu(3u,\tau,3\tau)+q^{-\frac{2}{3}}x^2\mu(3u,2\tau,3\tau),
\end{gather*}
where $\mu$ is the $\mu$-function defined by Zwegers as follows \cite{Zw1}:
\begin{gather*}
\mu(u,v;\tau)
 :=\frac{{\rm e}^{\pi{\rm i} u}}{\vartheta_{11}(v)}\sum_{n\in\Zz}\frac{(-1)^{n}{\rm e}^{2\pi{\rm i}nv}q^\frac{n(n+1)}{2}}{1-{\rm e}^{2\pi{\rm i}u}q^n}.
\end{gather*}
For convenience, we also use the following multiplicative notation of the $\mu$-function:
\begin{gather*}
\mu(x,y;q)=\mu(x,y):=\frac{{\rm i}q^{-\frac{1}{8}}\sqrt{xy}}{\theta_q(-y)}\sum_{n\in\Zz}\frac{(-1)^ny^nq^\frac{n(n+1)}{2}}{1-xq^n}.
\end{gather*}
If we substitute $x={\rm e}^{2\pi{\rm i}u}$ and $y={\rm e}^{2\pi{\rm i}v}$, then $\mu(x,y;q )=\mu(u,v;\tau )$.

Zwegers showed that the $\mu$-function satisfies a transformation law like Jacobi forms by adding an appropriate non-holomorphic function to the $\mu$-function \cite{Zw1}:
\begin{gather*}
\tilde{\mu}(u,v;\tau+1)
 =
 {\rm e}^{-\frac{\pi{\rm i}}{4}}\tilde{\mu}(u,v;\tau), \\
\tilde{\mu}\left(\frac{u}{\tau},\frac{v}{\tau};-\frac{1}{\tau}\right)
 =
 -{\rm i}\sqrt{-{\rm i}\tau}{\rm e}^{\pi{\rm i}\frac{(u-v)^2}{\tau}}\tilde{\mu}(u,v;\tau),
\end{gather*}
where
\begin{gather*}
\tilde{\mu}(u,v;\tau)
 :=
 \mu(u,v;\tau)+\frac{{\rm i}}{2}R(u-v;\tau), \qquad
E(x)
 :=
 2\int_0^x {\rm e}^{-\pi z^2}{\rm d}z, \nonumber \\
R(u;\tau)
 :=
 \sum_{\nu\in\Zz+\frac{1}{2}}\big\{{\rm sgn}(\nu)-E\big((\nu+a)\sqrt{2t}\big)\big\}(-1)^{\nu-\frac{1}{2}}{\rm e}^{-\pi{\rm i}\nu^2\tau-2\pi{\rm i}\nu u}, \nonumber
\end{gather*}
and $t={\rm Im}(\tau)$, $a=\frac{{\rm Im}(u)}{{\rm Im}(\tau)}$.
This result was a pioneering work in the study of mock modular forms (see~\cite{BFOR}).
Thus, the $\mu$-function is very important for the study of mock theta functions.

Further, Zwegers gave the following formulas:
\begin{gather}
\label{eq:mu periodicity}
\mu(u+1,v)
 =
 \mu(u,v+1)=-\mu(u,v), \\
\label{eq:mu pseudo periodicity}
\mu(u+\tau,v)
 =
 -{\rm e}^{2\pi{\rm i}(u-v)}q^\frac{1}{2}\mu(u,v)-{\rm i}{\rm e}^{\pi{\rm i}(u-v)}q^\frac{3}{8}, \\
\label{eq:mu translation}
\mu(u+z,v+z)
 =
 \mu(u,v)+\frac{{\rm i}q^{\frac{1}{8}}(q)_{\infty }^{3}\vartheta_{11}(z)\vartheta_{11}(u+v+z)}{\vartheta_{11}(u)\vartheta_{11}(v)\vartheta_{11}(u+z)\vartheta_{11}(v+z)}, \\
\label{eq:mu symmetry}
\mu(u,v)
 =
 \mu(u+\tau,v+\tau) \\
\label{eq:mu symmetry2}
 \hphantom{\mu(u,v)}{} =
 \mu(v,u) \\
\label{eq:mu symmetry3}
 \hphantom{\mu(u,v)}{} =
 \mu(-u,-v).
\end{gather}

On the other hand, the $\mu$-function is also a very interesting object from the viewpoint of $q$-hypergeometric functions.
For example, let us recall the well-known Kronecker formula \cite{W}:
\begin{gather}
\label{eq:Kronecker formula}
k(x,y)
 :=
 \frac{1}{1-x}\, {_{1}\psi _1}\left(\begin{matrix} x \\ qx \end{matrix};q,y\right)
 =
 \frac{(q,q,xy,q/xy)_{\infty}}{(x,q/x,y,q/y)_{\infty}}
 =
 \frac{(q)_{\infty }^{3}\theta _{q}(-xy)}{\theta _{q}(-x)\theta _{q}(-y)}.
\end{gather}
The second equality is the case of $a=x$, $b=xq$, $z=y$ in Ramanujan's summation formula:
\begin{gather*}
{_{1}\psi _1}\left(\begin{matrix} a \\ b \end{matrix};q,z\right)
 =
 \sum_{n \in \mathbb{Z}}
 \frac{(a)_{n}}{(b)_{n}}z^{n}
 =
 \frac{(az,q/az,q,b/a)_{\infty}}{(z,b/az,b,q/a)_{\infty}}.
\end{gather*}
Note that from the most right-hand side of (\ref{eq:Kronecker formula}), one recognizes the symmetry $k(x,y)=k(y,x)$ explicitly.

Obviously, the $\mu$-function is an analogue of the Kronecker summation:
\[
k(x,y)
 =
 \sum_{n\in\Zz}\frac{y^{n}}{1-xq^{n}}.
\]
Hence the symmetric property $\mu(x,y)=\mu(y,x)$ holds in a similar way as $k(x,y)=k(y,x)$, that is $\mu$-function has the following expressions:
\begin{align}
\label{eq:sym mu func 1}
\mu(x,y)
 &=
 \frac{{\rm i}q^{-\frac{1}{8}}\sqrt{xy}}{\theta_q(-x)}\frac{1}{1-y}
 \,{}_1\psi_2\left(\begin{matrix} y\\0,qy\end{matrix};q,qx\right) \\
\label{eq:sym mu func 2}
 &=
 \frac{{\rm i}q^{-\frac{1}{8}}\sqrt{xy}}{(x,y)_{\infty}}
 \, {}_2\psi_2\left(\begin{matrix}x,y\\0,0\end{matrix};q,q\right) \\
\label{eq:sym mu func 3}
 &=
 \frac{{\rm i}q^{-\frac{1}{8}}\sqrt{xy}}{(q/x,q/y)_{\infty}}\frac{1}{(1-x)(1-y)}
 \, {}_0\psi_2\left(\begin{matrix}-\\qx,qy\end{matrix};q,qxy\right).
\end{align}
These expressions follow from some Bailey transformations for ${_{2}\psi _2}$ (see \cite[p.~150, Exercise~5.20]{GR}):
\begin{align}
{_{2}\psi _2}\left(\begin{matrix}a,b \\ c,d \end{matrix};q,z\right)
\label{eq:Bailey trans0}
 &=
 \frac{(az,c/a,d/b,qc/abz)_{\infty}}{(z,c,q/b,cd/abz)_{\infty}}
 {_{2}\psi _2}\left(\begin{matrix}a,abz/c \\ az,d \end{matrix};q,\frac{c}{a}\right) \\
\label{eq:Bailey trans1}
 &=
 \frac{(bz,d/b,c/a,qd/abz)_{\infty}}{(z,d,q/a,cd/abz)_{\infty}}
 {_{2}\psi _2}\left(\begin{matrix}b,abz/d \\ bz,c \end{matrix};q,\frac{d}{b}\right) \\
\label{eq:Bailey trans2}
 &=
 \frac{(az,d/a,c/b,qd/abz)_{\infty}}{(z,d,q/b,cd/abz)_{\infty}}
 {_{2}\psi _2}\left(\begin{matrix}a,abz/d \\ az,c \end{matrix};q,\frac{d}{a}\right) \\
\label{eq:Bailey trans3}
 &=
 \frac{(bz,c/b,d/a,qc/abz)_{\infty}}{(z,c,q/a,cd/abz)_{\infty}}
 {_{2}\psi _2}\left(\begin{matrix}b,abz/c \\ bz,d \end{matrix};q,\frac{c}{b}\right),
\end{align}
which is a bilateral version of Heine's transformation formula:
\begin{gather*}
{_{2}\phi _1}\left(\begin{matrix}a,b \\ c \end{matrix};q,z\right)
 =
 \frac{(az,c/a)_{\infty}}{(z,c)_{\infty}}\,
 {_{2}\phi _1}\left(\begin{matrix}a,abz/c \\ az \end{matrix};q,\frac{c}{a}\right).
\end{gather*}
In fact, by taking the limit $z \rightarrow z/b$, $b\rightarrow \infty$, $c=0$ in (\ref{eq:Bailey trans0})--(\ref{eq:Bailey trans3}), we derive
\begin{align}
\label{eq:Bailey trans special1}
{}_1\psi_2\left(\begin{matrix}a\\0,d\end{matrix};q,z\right)
 &=
 \frac{(z,dq/az)_\infty}{(d,q/a)_\infty}\,{}_1\psi_2\left(\begin{matrix}az/d\\0,z\end{matrix};q,d\right) \\
\label{eq:Bailey trans special2}
 &=
 \frac{(d/a,dq/az)_\infty}{(d)_\infty}\,{}_2\psi_2\left(\begin{matrix}a,az/d\\0,0\end{matrix};q,\frac{d}{a}\right) \\
\label{eq:Bailey trans special3}
 &=
 \frac{(z,d/a)_{\infty}}{(q/a)_{\infty}}\,{}_0\psi_2\left(\begin{matrix}- \\z,d\end{matrix};q,az\right),
\end{align}
and we obtain (\ref{eq:sym mu func 1}), (\ref{eq:sym mu func 2}) and (\ref{eq:sym mu func 3}).
In particular, the expression (\ref{eq:sym mu func 3}) was pointed out by Choi (see \cite[Theorem 4]{C} and \cite[Theorem 5.1]{B}).

Thus the $\mu$-function is an interesting object deeply studied.
But, it seems that there are many problems to be clarified further.
For example, why the $\mu$-function is a two-variable function, whereas the universal mock theta function $g_{3}(x,q)$ is a one-variable function (i.e., what is the origin of the extra variable of the $\mu$-function?).
Also, why the factors~${\rm e}^{\pi{\rm i}u}$ in the numerator and~$\vartheta_{11}(v)$ in the denominator are needed.
Further, where the translation formula~(\ref{eq:mu translation}) comes from.
Some explanation has been given for these questions.
For example, Zwegers gives constructions of mock theta functions as indefinite theta series, of which $\mu$-function is a special case, and these naturally have theta functions in denominators which are explainable from the geometry of the quadratic forms. Also Zwegers' two-variable $\tilde{\mu}$-function is decomposed into a sum of meromorphic Jacobi form and one-variable Maa\ss--Jacobi form~\cite{R} (see also \cite[Theorem~8.18]{BFOR}).
However, as we will show in this paper, one can give another answer from the view of $q$-special functions.

By using the pseudo-periodicity equation (\ref{eq:mu pseudo periodicity}) and a bit of calculation, we obtain the following second-order $q$-difference equation for the $\mu$-function:
\begin{gather}
\label{eq:mu q-diff}
\left[T_{x}^2-q^{\frac{1}{2}}\left(1-\frac{x}{y}q\right)T_{x}-\frac{x}{y}q\right]\mu(x,y)=0, \qquad T_{x}f(x):=f(qx).
\end{gather}
This $q$-difference equation (\ref{eq:mu q-diff}) coincides with the specialization $a=q$ and the transformation $x\mapsto \frac{x}{y}$ of the following equation essentially:
\begin{gather}
\label{eq:q-Hermite}
\big[T_x^2-(1-xq)\sqrt{a}T_x-xq\big]f(x)=0,
\end{gather}
which is the transformation $a\mapsto \frac{q}{a}$, $u(xq)=\theta_q\big({-}\frac{\sqrt{a}}{xq}\big)f(x)$ of the $q$-Hermite--Weber equation:
\begin{gather*}
\big[axT_{x}^{2}+(1-x)T_{x}-1\big]u(x)=0.
\end{gather*}
The equation (\ref{eq:q-Hermite}), which we also call the ``$q$-Hermite--Weber equation'', is a typical example of the second-order $q$-difference equation of the Laplace type:
\begin{gather*} 
\big[(a_{0}+b_{0}x)T_{x}^{2}+(a_{1}+b_{1}x)T_{x}+(a_{2}+b_{2}x)\big]u(x)=0.
\end{gather*}

The $q$-difference equations of the Laplace type have long been studied.
For example, the $q$-difference equation satisfied by Heine's hypergeometric function ${}_2\phi_1$:
\begin{gather}
\label{eq:q-Heine diff}
\big[(c-abqx)T_{x}^{2}-(c+q-(a+b)qx)T_x+q(1-x)\big]f(x)=0,
\end{gather}
which is the most popular and master class in the following hierarchy of the second order $q$-difference equations, was discovered in the 19th century.
The $q$-Hermite--Weber equation~(\ref{eq:q-Hermite}) is one of some equations in the hierarchy obtained by taking some limits of $q$-difference equation~(\ref{eq:q-Heine diff}) (for example, see \cite[Figure~2]{O1}).

However, a systematic study of the global analysis of degenerate Laplace types had to wait for J.P.~Ramis, J.~Sauloy and C.~Zhang's work~\cite{RSZ} in the 21st century.
They introduced some $q$-Borel and $q$-Laplace transformations:
\begin{gather}
\mathcal{B}^{+}(f)(\xi):=\sum_{n\geq 0}a_{n}q^{\frac{n(n-1)}{2}}\xi^{n}, \qquad
\label{eq:q-Laplace}
 \mathcal{L}^{+}(f)(x,\lambda):=\sum_{n\in\Zz}\frac{f(\lambda q^n)}{\theta_q(\lambda q^{n}/x)},
\end{gather}
for the formal series
\[
f(x)=\sum_{n\geq 0}a_{n}x^{n} \in \Cz\llbracket x \rrbracket,
\]
which play a fundamental role in the study of the Laplace type $q$-difference equations.
In particular, C.~Zhang \cite{Zh} obtained some connection formulas for the $q$-convergent hypergeometric series ${}_2\phi_0(a,b;q,x)$ by applying these transformations.
By considering the degeneration of Zhang's result, C.~Zhang and Y.~Ohyama \cite{O2} gave some connection formulas for the $q$-Hermite series ${}_1\phi_1(a;0;q,x)$.

In view of the story discussed above, we introduce the following generalization of the $\mu$-function.
\begin{dfn}\label{def:mua}
Let $\alpha $, $a$ be complex parameters such that $u-\alpha\tau,v\in\Cz\backslash(\Zz+\Zz\tau)$, and such that $xa^{-1}, y \in \mathbb{C} \backslash \{q^{n}\}_{n \in \mathbb{Z}}$.
We define $\mu(u,v;\alpha)$ or $\mu(x,y;a)$ as the following series:
\begin{gather*}
\mu(u,v;\alpha,\tau)
 =
 \mu(u,v;\alpha)
 :=
 \frac{{\rm e}^{\pi{\rm i}\alpha(u-v)}}{\vartheta_{11}(v)}
 \sum_{n\in\Zz}(-1)^n{\rm e}^{2\pi{\rm i}(n+\frac{1}{2})v}q^\frac{n(n+1)}{2}
 \frac{\big({\rm e}^{2\pi{\rm i}u}q^{n+1}\big)_{\infty }}{\big({\rm e}^{2\pi{\rm i}u}q^{n-\alpha +1}\big)_{\infty }}, \\
\mu(x,y;a,q)
 =
 \mu(x,y;a)
 :=
 -{\rm i}q^{-\frac{1}{8}}\frac{(x/y)^\frac{\alpha}{2}}{\theta_q(-y)}\frac{(x)_\infty}{(x/a)_\infty}
 {}_1\psi_2\left(\begin{matrix}x/a\\0,x\end{matrix};q,y\right).
\end{gather*}
In particular, we have $\mu(u,v;0)=-{\rm i}q^{-\frac{1}{8}}$ and $\mu(u,v;1)=\mu(u,v)$.
\end{dfn}
If we substitute $x={\rm e}^{2\pi{\rm i}u}$, $y={\rm e}^{2\pi{\rm i}v}$ and $a={\rm e}^{2\pi{\rm i}\alpha \tau }=q^{\alpha }$, then $\mu(x,y;a,q)=\mu(u,v;\alpha,\tau )$.
The definition of $\mu(u,v;\alpha)$ is equal to the composition of the $q$-Borel and $q$-Laplace transformations of the formal solution
\[
\widetilde{f}_{0}(x)
 =
 {}_2\phi_0\left(\begin{matrix}a,0 \\ - \end{matrix};q, {\frac{x}{a}}\right)
\]
for the $q$-Hermite--Weber equation (\ref{eq:q-Hermite}) around $x=0$ (see Section~\ref{section2}).
\begin{thm}\label{thm:mu and q-Hermite--Weber}
Let $f_{0}(x,\lambda )$ be the image of the $q$-Borel and $q$-Laplace transformations of the fundamental solution $\widetilde{f}_{0}$ of the $q$-Hermite--Weber equation at $x=0$:
\[
f_{0}(x)
 :=
 x^{\frac{\alpha }{2}}\mathcal{L}^{+}\circ\mathcal{B}^{+}\big(\widetilde{f}_{0}\big)(x,\lambda).
\]
Then we have
\begin{gather}\label{eq:mu and q-Hermite--Weber}
f_{0}\big({\rm e}^{2\pi{\rm i}(u-v)},-{\rm e}^{2\pi{\rm i}u}\big)
 =\mu(u,v;\alpha).
\end{gather}
\end{thm}
As a corollary, we obtain an interpretation of $\mu(u,v)=\mu(u,v;1)$ that the original $\mu$-function is regarded as the image of the $q$-Borel and $q$-Laplace transformations of the fundamental solution~$\widetilde{f}_{0}$ of the $q$-Hermite--Weber equation
(very recently S.~Garoufalidis and C.~Wheeler~\cite{GW} pointed out this fact independently of us).
This interpretation gives some answers to our questions mentioned before.
For example, the $\mu$-function is a two-variable function with $x={\rm e}^{2\pi{\rm i} u}$ and $y={\rm e}^{2\pi{\rm i} v}$ due to the extra parameter $\lambda $ which arises from the $q$-Laplace transformation~(\ref{eq:q-Laplace}).
The denominator $\vartheta_{11}(v)$ of the $\mu$-function arises from the definition of $q$-Laplace transformation.
The numerator ${\rm e}^{\pi{\rm i} u}$ in $\mu(u,v)$ corresponds to the characteristic index of the solution around the origin of~(\ref{eq:q-Hermite}).

For this function $\mu(u,v;\alpha)$, we give the following formulas similar to those of the original $\mu$-function.
\begin{thm}\label{thm:Thm1}
\begin{gather}
\label{eq:main results0}
\mu(u+2\tau,v;\alpha)
 =
 \big(1-{\rm e}^{2\pi{\rm i} (u-v)}q\big)q^\frac{\alpha}{2}\mu(u+\tau,v;\alpha)+{\rm e}^{2\pi{\rm i}(u-v)}q\mu(u,v;\alpha), \\
\label{eq:main results1}
\mu(u,v;\alpha)
 =
 {\rm e}^{-\pi{\rm i} \alpha}\mu(u+1,v;\alpha)
 =
 {\rm e}^{\pi{\rm i}\alpha}\mu(u,v+1;\alpha), \\
\label{eq:main results2}
\mu(u+\tau,v;\alpha)
 =
 -{\rm e}^{2\pi{\rm i}(u-v)}q^\frac{\alpha}{2}\mu(u,v;\alpha)
 +{\rm e}^{\pi{\rm i}(u-v)}q^\frac{\alpha}{2}\mu(u,v;\alpha-1), \\
\label{eq:main results2-2}
\mu(u-\tau,v;\alpha)
 =
 q^\frac{\alpha}{2}\mu(u,v;\alpha)
 -2{\rm i}{\rm e}^{-\pi{\rm i}(u-v)}\sin(\pi\alpha\tau)\mu(u,v;\alpha+1), \\
\label{eq:main results3}
\mu(u+z,v+z;\alpha)
 =
 \frac{\vartheta_{11}(u+z)\vartheta_{11}(v+z-\alpha\tau)}{\vartheta_{11}(u+z-\alpha\tau)\vartheta_{11}(v+z)}{\rm e}^{2\pi{\rm i}\alpha(u-v)}\mu(v,u;\alpha) \nonumber \\
\hphantom{\mu(u+z,v+z;\alpha)=}{}
 -
 \frac{i(q^\alpha)_\infty(q)_\infty^2q^\frac{1-4\alpha}{8}\vartheta_{11}(z)\vartheta_{11}(u+v+z-\alpha\tau)}{\vartheta_{11}(u)\vartheta_{11}(v-\alpha\tau)\vartheta_{11}(u+z-\alpha\tau)\vartheta_{11}(v+z)} \nonumber \\
\hphantom{\mu(u+z,v+z;\alpha)=}{}\times {\rm e}^{\pi{\rm i}(\alpha-1)(u-v)}\, {}_1\phi_1\hyper{q^{1-\alpha}}{0}{q}{{\rm e}^{-2\pi{\rm i}(u-v)}q}, \\
\label{eq:main results4}
\mu(u,v;\alpha)
 =
 \mu(u+\tau,v+\tau;\alpha) \\
\label{eq:main results4-2}
 \hphantom{\mu(u,v;\alpha)}{} =
 \frac{\vartheta_{11}(v-\alpha\tau)\vartheta_{11}(u)}{\vartheta_{11}(u-\alpha\tau)\vartheta_{11}(v)}{\rm e}^{2\pi{\rm i}\alpha(u-v)}\mu(v,u;\alpha) \\
 \hphantom{\mu(u,v;\alpha)}{} =
\label{eq:main results4-3}
 \frac{\vartheta_{11}(v-\alpha\tau)\vartheta_{11}(u)}{\vartheta_{11}(u-\alpha\tau)\vartheta_{11}(v)}{\rm e}^{2\pi{\rm i}\alpha(u-v)}\mu(-u+\alpha\tau,-v+\alpha\tau;\alpha), \\
\label{eq:main results5}
2\cos\pi(u-v)\mu(u,v;\alpha)
 =
 (1-q^{-\alpha})\mu(u,v;\alpha+1)+\mu(u,v;\alpha-1).
\end{gather}
\end{thm}

We see that these formulas correspond to that of the original $\mu$-function as the periodicity: $(\ref{eq:main results1})\leftrightarrow(\ref{eq:mu periodicity})$,
forward shift: $(\ref{eq:main results2})\leftrightarrow(\ref{eq:mu pseudo periodicity})$,
translation: $(\ref{eq:main results3})\leftrightarrow(\ref{eq:mu translation})$,
$\tau$-periodicity: $(\ref{eq:main results4})\leftrightarrow(\ref{eq:mu symmetry})$,
symmetry: $(\ref{eq:main results4-2})\leftrightarrow(\ref{eq:mu symmetry2})$,
pseudo periodicity: $(\ref{eq:main results4-3})\leftrightarrow(\ref{eq:mu symmetry3})$.
The equation (\ref{eq:main results0}) is a rewriting that $\mu(u,v;\alpha)$ satisfies the $q$-Hermite--Weber equation (\ref{eq:q-Hermite}).
Also, the property~(\ref{eq:main results5}) is one of the specific properties of $\mu(u,v;\alpha)$, which coincides essentially with the $q$-Bessel equation:
\[
\left[T_x-\big(q^\frac{\nu}{2}+q^{-\frac{\nu}{2}}\big)T_x^\frac{1}{2}+\left(1+\frac{x^2}{4}\right)\right]f(x)=0, \qquad T_{x}^{\frac{1}{2}}f(x):=f\big(q^{\frac{1}{2}}x\big).
\]
Also, the second term on the right-hand side of (\ref{eq:main results3}) is written by the $q$-Bessel function (see Corollary \ref{rmk q-Bessel}).
Further, we prove the translation formula (\ref{eq:main results3}) as a connection formula for the $q$-Hermite--Weber equation, that means the mysterious translation (\ref{eq:main results3}) is regarded as a~variation of a~connection formula (see Theorem \ref{thm:tuchimi connection} and (\ref{eq:tuchimi connection})).

We also immediately obtain the $q$-hypergeometric expressions for $\mu(u,v;\alpha)$ corresponding to (\ref{eq:sym mu func 1}), (\ref{eq:sym mu func 2}) and (\ref{eq:sym mu func 3}) from (\ref{eq:Bailey trans special1}), (\ref{eq:Bailey trans special2}) and (\ref{eq:Bailey trans special3}).
\begin{thm}\label{Theorem 3}
\begin{align*}
\mu(x,y;a)
 &=
 -{\rm i}q^{-\frac{1}{8}}\frac{(x/y)^\frac{\alpha}{2}}{\theta_q(-x/a)}\frac{(aq/y)_{\infty }}{(q/y)_{\infty }}{}_1\psi_2\left(\begin{matrix}y/a\\0,y\end{matrix};q,x\right) \\
 &=
 -{\rm i}q^{-\frac{1}{8}}\left(\frac{x}{y}\right)^\frac{\alpha}{2}\frac{(a,q,aq/x,aq/y)_{\infty }}{\theta_{q}(-y)\theta_{q}(-x/a)}{}_2\psi_2\left(\begin{matrix}x/a,y/a\\0,0\end{matrix};q,a\right) \\
 &=
 -{\rm i}q^{-\frac{1}{8}}\left(\frac{x}{y}\right)^\frac{\alpha}{2}\frac{(a,q,x,y)_{\infty }}{\theta_{q}(-y)\theta_{q}(-x/a)}{}_0\psi_2\left(\begin{matrix}-\\x,y\end{matrix};q,\frac{xy}{a}\right).
\end{align*}
\end{thm}
Note that the symmetry (\ref{eq:main results4-2}) is proved by a specialization of the translation (\ref{eq:main results3}), but we also get another proof from these $q$-hypergeometric expressions.

The above results are for general complex parameters with $\alpha$, but by restricting $\alpha$ to the integer $k$, we obtain the following simplified formulation.
\begin{cor}\label{cor:Thm1 k}
Let $k$ be an integer, we have
\begin{gather}
\mu(u+1,v;k)
 =
 \mu(u,v+1;k)
 =
 (-1)^k\mu(u,v;k), \nonumber\\
\mu(u+\tau,v;k)
 =
 -{\rm e}^{2\pi{\rm i}(u-v)}q^\frac{k}{2}\mu(u,v;k)+{\rm e}^{\pi{\rm i}(u-v)}q^\frac{k}{2}\mu(u,v;k-1), \nonumber\\
\mu(u-\tau,v;k)
 =
 q^\frac{k}{2}\mu(u,v;k)
 -2{\rm i}{\rm e}^{-\pi{\rm i}(u-v)}\sin(\pi k \tau)\mu(u,v;k+1), \nonumber\\
\mu(u+z,v+z;k+1)
 =\mu(u,v;k+1) +\frac{{\rm i} q^{\frac{1}{8}}(q)_{\infty }^{3}\vartheta_{11}(z)\vartheta_{11}(u+v+z)}{\vartheta_{11}(u)\vartheta_{11}(v)\vartheta_{11}(u+z)\vartheta_{11}(v+z)} \nonumber \\
\hphantom{\mu(u+z,v+z;k+1)=}{}
\times \frac{{\rm e}^{-\pi{\rm i}k(u-v)}}{(q^{-k})_k}\,{}_1\phi_1\left(\begin{matrix}q^{-k}\\0\end{matrix};q,{\rm e}^{2\pi{\rm i}(u-v)}q\right), \nonumber\\
\mu(u,v;k) =
 \mu(u+\tau,v+\tau;k)\nonumber \\
\hphantom{\mu(u,v;k)}{}
=
 \mu(v,u;k) \nonumber\\
\hphantom{\mu(u,v;k)}{} =
 \mu(-u,-v;k), \nonumber\\
\label{eq:mu k 5}
2\cos\pi(u-v)\mu(u,v;k)
 =
 \big(1-q^{-k}\big)\mu(u,v;k+1)+\mu(u,v;k-1).
\end{gather}
In particular, if $k$ is a positive integer, we have
\begin{align}
\mu(u,v;k)
 &=
 \frac{{\rm e}^{\pi{\rm i}k(u-v)}}{\vartheta_{11}(v)}\sum_{n\in\Zz}(-1)^n{\rm e}^{2\pi{\rm i}\left(n+\frac{1}{2}\right)v}q^\frac{n(n+1)}{2}\prod_{l=0}^{k-1}\frac{1}{1-{\rm e}^{2\pi{\rm i}u}q^{n-l}} \nonumber\\
\label{eq:mu k 7}
 &=
 {\rm e}^{-\pi{\rm i}(k-1)(u-v-\tau)}\sum_{j=0}^{k-1}\frac{(-1)^{k-1-j}}{(q)_j(q)_{k-1-j}}q^\frac{(k-1-j)^2}{2}\mu(u-j\tau,v).
\end{align}
\end{cor}
Also from (\ref{eq:mu k 5}), we obtain the following result which is non-trivial from the definition of $\mu(u,v;\alpha )$.
\begin{thm}\label{thm:mu and CqH}
Let $k$ be a non-negative integer, we have
\begin{gather*}
\mu(u,v;-k)
 =
 -{\rm i}q^{-\frac{1}{8}}H_{k}(\cos{\pi (u-v)}\mid q),
\end{gather*}
where $H_{k}(\cos{\pi (u-v)}\mid q)$ is a continuous $q$-Hermite polynomial of degree $k$ {\rm \cite[p.~543]{KLS}}:
\[
H_{k}(\cos{\pi (u-v)} \mid q)
 :=
 \sum_{l=0}^{k}
 \frac{(q)_{k}}{(q)_{l}(q)_{k-l}}
 {\rm e}^{\pi{\rm i}(k-2l)(u-v)}.
\]
\end{thm}
From this theorem, $\{\mu(u,v;k+1)\}_{k\geq 0}$ is a family of ``minus degree'' continuous $q$-Hermite polynomials, and in particular, the original $\mu$-function is regarded as a ``$-1$ degree'' continuous $q$-Hermite polynomial.

To give some relations of $\{\mu(u,v;k)\}_{k \in \mathbb{Z}}$, we introduce the following generating function of $\{\mu(u,v;k+1)\}_{k\geq 0}$:
\begin{gather*}
S(u,v,r)=S(r):=\sum_{k=0}^{\infty }\mu(u,v;k+1)r^{k},
\end{gather*}
which is a variation (i.e., minus degree version) of the well-known generating function of the continuous $q$-Hermite polynomials (for example, see \cite[p.~542]{KLS}):
\begin{gather}
\label{eq:gen func of CqH}
\sum_{n\geq0}\frac{H_n(\cos\theta \mid q)}{(q)_n}r^n=\frac{1}{\big(r{\rm e}^{{\rm i}\theta},r{\rm e}^{-{\rm i}\theta}\big)_\infty}.
\end{gather}
For the generating function $S(r)$, we obtain the following $q$-difference relations and expressions.
\begin{thm}\label{thm:Sr}\quad
\begin{enumerate}\itemsep=0pt
\item[$(1)$] The generating function $S(r)$ satisfies the following $q$-difference equations:
\begin{gather}
\label{eq:Sr rec1}
S(r)
 =
 \big(1-r{\rm e}^{\pi{\rm i}(u-v)}q\big)\big(1-r{\rm e}^{-\pi{\rm i}(u-v)}q\big)S(rq)-{\rm i}rq^\frac{7}{8}, \\
 \big[\big(1-r{\rm e}^{\pi{\rm i}(u-v)}q^{2}\big)\big(1-r{\rm e}^{-\pi{\rm i}(u-v)}q^{2}\big)T_r^2 \nonumber \\
\label{eq:Sr rec2}
\qquad{} -
 \big(1+q\big(1-r{\rm e}^{\pi{\rm i}(u-v)}q\big)\big(1-r{\rm e}^{-\pi{\rm i}(u-v)}q\big)\big)T_r+q\big]S(r)=0.
\end{gather}
\item[$(2)$] The generating function $S(r)$ has the following closed forms:
\begin{gather}
S(r)
 =
 \big(r{\rm e}^{\pi{\rm i}(u-v)}q,r{\rm e}^{-\pi{\rm i}(u-v)}q\big)_{\infty }
 \mu(u,v)-{\rm i}rq^{\frac{7}{8}}\,
 {_{3}\phi _2}\left(\begin{matrix}q, r{\rm e}^{\pi{\rm i}(u-v)}q, r{\rm e}^{-\pi{\rm i}(u-v)}q \\ 0,0 \end{matrix};q,q\right) \nonumber\\
\hphantom{S(r)}{}
=
 \big(r{\rm e}^{\pi{\rm i}(u-v)}q,r{\rm e}^{-\pi{\rm i}(u-v)}q\big)_{\infty } \nonumber \\
\hphantom{S(r)=}{}
 \times \left\{\mu(u,v)-\frac{{\rm i}rq^{\frac{7}{8}}}{1-q}\Phi^{(1)}\left(\begin{matrix}a;0,0 \\ q^{2} \end{matrix};q;r{\rm e}^{\pi{\rm i}(u-v)}q, r{\rm e}^{-\pi{\rm i}(u-v)}q\right)\right\}\label{eq:Sr expression2} \\
\label{eq:Sr expression3}
\hphantom{S(r)}{}
 =
 \big(r{\rm e}^{\pi{\rm i}(u-v)}q,r{\rm e}^{-\pi{\rm i}(u-v)}q\big)_{\infty }
 \sum_{m\geq 0}\frac{\mu(u,v;1-m )}{(q)_{m}}q^{m}r^{m},
\end{gather}
where $\Phi^{(1)}$ is the $q$-Appell hypergeometric function:
\begin{gather}
\label{q-Appell}
\Phi^{(1)}\left(\begin{matrix}a;b_{1},b_{2} \\ c \end{matrix};q;x, y\right)
 :=
 \sum_{m,n\geq0}
 \frac{(a)_{m+n}(b_1)_m(b_2)_n}{(c)_{m+n}(q)_m(q)_n}x^{m}y^{n}.
\end{gather}
\end{enumerate}
\end{thm}
In particular, note that the $q$-difference equation (\ref{eq:Sr rec2}) is a degeneration of the $q$-Heun equation:
\[
\big[\big(a_{0}+b_{0}x+c_{0}x^{2}\big)T_{x}^{2}+\big(a_{1}+b_{1}x+c_{1}x^{2}\big)T_{x}+\big(a_{2}+b_{2}x+c_{2}x^{2}\big)\big]u(x)=0,
\]
and $S(r)$ is its solution.
Furthermore, by comparing the coefficients in the expansion of the relation~(\ref{eq:Sr expression3}) at $r=0$, we also obtain some relations for $\mu(u,v;k+1)$, see Corollary~\ref{cor:Sr expressions}.

This paper is organized as follows.
First, in Section~\ref{section2}, we present some preliminary results on the basic solutions and connection formulas for some $q$-difference equations that are degenerations of the $q$-difference equation~(\ref{eq:q-Heine diff}). In particular, for the $q$-Hermite--Weber equation (\ref{eq:q-Hermite}), we give a proof of its connection formula according to Ohyama's private note~\cite{O2}.
Furthermore, we show that the images of the $q$-Borel and $q$-Laplace transformations for the divergent solution $\widetilde{f}_{0}(x)$ is essentially equivalent to the generalized $\mu$-function $\mu(u,v;\alpha)$.
Then, we prove the connection formula for the case where the parameters $\lambda $ and $\lambda^{\prime}$ are different, and state that it is equivalent to the translation formula (\ref{eq:main results3}) of $\mu(u,v;\alpha)$.
Next, in Section~\ref{section3}, we prove the above main results concerning $\mu(u,v;\alpha)$.
In Section~\ref{section4}, we mention that $\mu(u,v;k,\tau)$ satisfies the modular transformations (\ref{eq:tnu +1}) and (\ref{eq:tnu +tau}).
Finally, in Section~\ref{section5}, we discuss a possible direction to future works.

During the preparation of this paper, we have been informed by Hikami and Matsusaka of a~recent very interesting paper~\cite{GW}.
Some of the examples of the paper \cite{GW} overlap with our results on the $\mu$-function, though the one parameter generalization $\mu(u,v;\alpha)$ is not considered there (e.g., see equation (\ref{eq:mu q-diff}) and \cite[Section~4.2]{GW}).

\section[Fundamental solutions and connection formulas of some q-difference equations]{Fundamental solutions and connection formulas\\ of some $\boldsymbol{q}$-difference equations}\label{section2}

In this section, based on \cite{O1, RSZ,Zh}, we present some preliminary results on the basic solutions and connection formulas for some $q$-difference equations.
\begin{lem}[{\cite[p.~18, Theorem~2.2.1]{Zh}}]
\label{lem:lemma1}
We have the fundamental solutions of the $q$-difference equation
\begin{gather}
\label{eq:(26)}
\big[(1-abxq)T_x^2-(1-(a+b)xq)T_x-xq\big]f(x)=0
\end{gather}
around $x=0$ and $x= \infty$:
\begin{gather*}
 \mathcal{L}^{+}\circ\mathcal{B}^{+}(F_1)(x,\lambda), \qquad F_2(x)=\frac{(abx)_\infty}{\theta_q(-xq)}\,{}_2\phi_1\left(\begin{matrix}q/a,q/b\\0 \end{matrix};q,abx\right), \\
 G_1(x)
 =
 \frac{\theta_q(-axq)}{\theta_q(-xq)}\,{}_2\phi_1\left(\begin{matrix}a,0\\ aq/b\end{matrix};q,\frac{q}{abx}\right), \qquad G_2(x)=\frac{\theta_q(-bxq)}{\theta_q(-xq)}\,{}_2\phi_1\left(\begin{matrix}b,0\\ bq/a\end{matrix};q,\frac{q}{abx}\right),
\end{gather*}
where $F_{1}(x)$ is the formal solution around $x=0$:
\[
F_1(x)
 :=
 {}_2\phi_0\left(\begin{matrix}a,b \\ -\end{matrix}q,x\right).
\]

In this case, the connection formulas for $\mathcal{L}^{+}\circ\mathcal{B}^{+}(F_1)(x,\lambda)$, $F_2(x)$ and $G_1(x)$, $G_2(x)$ are as follows:
\begin{gather*}
\mathcal{L}^{+}\circ\mathcal{B}^{+}(F_1)(x,\lambda)
 =
 \frac{(b)_\infty\theta_q(a\lambda)\theta_q(axq/\lambda)\theta_q(-xq)}{(b/a)_\infty\theta_q(\lambda,xq/\lambda,-axq)}G_1(x)\\
\hphantom{\mathcal{L}^{+}\circ\mathcal{B}^{+}(F_1)(x,\lambda)=}{}
 +\frac{(a)_\infty\theta_q(b\lambda)\theta_q(bxq/\lambda)\theta_q(-xq)}{(a/b)_\infty\theta_q(\lambda)\theta_q(xq/\lambda)\theta_q(-bxq)}G_2(x),\\
F_2(x)
 =
 \frac{(q/a)_\infty}{(b/a,q)_\infty}G_1(x)+\frac{(q/b)_\infty}{(a/b,q)_\infty}G_2(x).
\end{gather*}
\end{lem}
Putting $x\mapsto a^{-1}x$, $b=0$ in the $q$-difference equation $(\ref{eq:(26)})$, we have
\[
\left[T_x^2-(1-xq)T_x-\frac{x}{a}q\right]f\big(a^{-1}x\big)=0.
\]
Furthermore, we put $a=q^{\alpha }$, $x^\frac{\alpha}{2}f\big(a^{-1}x\big)=g(x)$, then $g(x)$ is the solution of the $q$-Hermite--Weber equation (\ref{eq:q-Hermite}).
\begin{lem}\label{lem:lemma2}
The fundamental solutions of the $q$-difference equation \eqref{eq:q-Hermite} around $x=0$ are
\begin{gather*}
f_{0}(x)
 :=
 x^{\frac{\alpha }{2}}\mathcal{L}^{+}\circ\mathcal{B}^{+}\big(\widetilde{f}_{0}\big)(x,\lambda), \qquad
g_0(x)
 :=
 \frac{x^{1-\frac{\alpha}{2}}}{\theta_q(-x)}\,{}_1\phi_1\left(\begin{matrix}q/a\\0 \end{matrix};q,xq\right),
\end{gather*}
where $\widetilde{f}_{0}(x)$ is the formal solution around $x=0$:
\[
\widetilde{f}_{0}(x)
 =
 {}_2\phi_0\left(\begin{matrix}a,0 \\ - \end{matrix};q, {\frac{x}{a}}\right).
\]
And that of around $x=\infty$ are
\begin{gather*}
f_{\infty }(x)
 :=
 f_{0 }\big(x^{-1}\big), \qquad
g_{\infty }(x)
 =g_0\big(x^{-1}\big).
\end{gather*}

In this case, the connection formulas for $f_0(x,\lambda)$, $f_{\infty }(x,\lambda)$ and $g_0(x)$, $g_\infty(x)$ are as follows:
\begin{align}
\label{amuconnect}
\begin{pmatrix}f_0(x,\lambda)\\ f_{\infty }(x,\lambda) \end{pmatrix}=-\frac{(q)_\infty}{(q/a)_\infty}\begin{pmatrix}\dfrac{\theta_q(\lambda)\theta_q(ax/\lambda)x^\alpha}{\theta_q(\lambda/a)\theta_q(x/\lambda)a}&1 \\ 1&\dfrac{\theta_q(\lambda)\theta_q(x\lambda/a)}{\theta_q(\lambda/a)\theta_q(x\lambda)}x^{-\alpha}\end{pmatrix}\begin{pmatrix}g_0(x)\\g_\infty(x)\end{pmatrix}.
\end{align}
\end{lem}
For this Lemma\,\ref{lem:lemma2}, we give a proof by following Ohyama \cite{O2}.
\begin{proof}
It is sufficient to prove the case of $x=0$, since the $q$-difference equation (\ref{eq:q-Hermite}) is symmetric under the transformation $x\leftrightarrow x^{-1}$.

The image of $\mathcal{L}^{+}\circ\mathcal{B}^{+}$ for $\widetilde{f}_{0}(x)$ gives the convergent series $f_{0}(x)$, and it gives a fundamental solution of (\ref{eq:q-Hermite}).
This fact follows from the property of the $q$-Borel and $q$-Laplace transformations:
\begin{gather*}
\mathcal{B}^{+}\big(x^mT_x^nf(x)\big)(\xi)=q^\frac{m(m-1)}{2}\xi^mT_\xi^{m+n}\mathcal{B}^+(f)(\xi), \\
\mathcal{L}^{+}\big(\xi^mT_\xi^nf(\xi)\big)(x,\lambda)=q^{-\frac{m(m-1)}{2}}x^mT_x^{n-m}\mathcal{L}^+(f)(x,\lambda),
\end{gather*}
or
\begin{gather*}
\mathcal{L}^{+}\circ\mathcal{B}^{+}\big(x^mT_x^nf\big)(x,\lambda)=x^mT_x^n\mathcal{L}^+\circ\mathcal{B}^+(f)(x,\lambda).
\end{gather*}
Since the limit of $F_{2}(x)$ as $x\to a^{-1}x$, $b\to 0$ degenerates to $g_{0}(x)$, we prove $g_{0}(x)$ is another fundamental solution of~(\ref{eq:q-Hermite}).

The proof of the connection formula is obtained from the degeneration limit of the connection formula in Lemma~\ref{lem:lemma1}.
We take the limit of the Heine's transformation formula
\[
{}_2\phi_1\left(\begin{matrix}a,b\\c\end{matrix};q,x\right)
 =
 \frac{(abx/c)_\infty}{(x)_\infty}\,{}_2\phi_1\left(\begin{matrix}c/b,c/a\\ c\end{matrix};q,\frac{abx}{c}\right),
\]
as $a\mapsto0$, then
\[
{}_2\phi_1\left(\begin{matrix}0,b\\c\end{matrix};q,x\right)=\frac{1}{(x)_\infty}\, {}_1\phi_1\left(\begin{matrix}q/b\\c\end{matrix};q,bx\right).
\]
Then we rewrite the solution $G_2(x)$ in Lemma~\ref{lem:lemma1} as
\begin{align*}
G_2(x)
 &=
 \frac{\theta_q(-bxq)}{\theta_q(-xq)}\,{}_2\phi_1\left(\begin{matrix}b,0\\ bq/a\end{matrix};q,\frac{q}{abx}\right)
 =
 \frac{\theta_q(-bxq)}{(q/abx)_\infty \theta_q(-xq)}\,{}_1\phi_1\left(\begin{matrix}q/a\\ bq/a\end{matrix};q,\frac{q}{ax}\right) \\
 &=
 (q,abx)_\infty \frac{\theta_q(-bxq)}{\theta_q(-abx)\theta_q(-xq)}\,{}_1\phi_1\left(\begin{matrix}q/a\\ bq/a\end{matrix};q,\frac{q}{ax}\right).
\end{align*}
Therefore the connection formula of $\mathcal{L}^+\circ\mathcal{B}^+(F_1)(x,\lambda)$ is given by
\begin{gather}
\label{eq:prot connec}
\mathcal{L}^+\circ\mathcal{B}^+(F_1)(x,\lambda)=C_1(x)F_2(x)+C_2(x)\frac{(abx)_\infty}{\theta_q(-ax)}\,{}_1\phi_1\left(\begin{matrix}q/a\\ bq/a\end{matrix};q,\frac{q}{ax}\right),
\end{gather}
where
\begin{gather*}
C_1(x)
 =
 \frac{(b,q)_\infty}{(q/a)_\infty}\frac{\theta_q(a\lambda)\theta_q(axq/\lambda)\theta_q(-xq)}{\theta_q(\lambda)\theta_q(xq/\lambda)\theta_q(-axq)}, \\
C_2(x)
 =
 \frac{(q)_\infty\theta_q(-ax)}{\theta_q(\lambda)\theta_q(\frac{xq}{\lambda})}\left\{ (a,bq/a,q)_\infty\frac{\theta_q(b\lambda)\theta_q(bxq/\lambda)}{\theta_q(-bq/a)\theta_q(-abx)}\right.\\
\left.\hphantom{C_2(x)=}{}
-\frac{(bq/a)_\infty}{(q/a)_\infty}\frac{\theta_q(-b)\theta_q(a\lambda)\theta_q(axq/\lambda)\theta_q(-bxq)}{\theta_q(-bq/a)\theta_q(-axq)\theta_q(-abx)}\right\}.
\end{gather*}

By replacing $x\mapsto a^{-1}x$, $\lambda\mapsto a^{-1}\lambda$, $b\mapsto-\lambda^{-1}q^n$ in (\ref{eq:prot connec}) and taking the limit of each term of~(\ref{eq:prot connec}) as $n\to\infty$, we have
\begin{gather*}
 \frac{(abx)_\infty}{\theta_q(-ax)}\,{}_1\phi_1\left(\begin{matrix}q/a\\ bq/a\end{matrix};q,\frac{q}{ax}\right)\to-x^{-\frac{\alpha}{2}}g_{\infty }(x), \\
 F_2(x)\to\frac{\theta_q(-x)}{\theta_q(-xq/a)}x^{\frac{\alpha}{2}-1}g_0(x), \qquad
 \mathcal{L}^{+}\circ\mathcal{B}^{+}(F_1)(x,\lambda)\to \mathcal{L}^{+}\circ\mathcal{B}^{+}\big(\widetilde{f}_{0}\big)(x,\lambda).
\end{gather*}
Hence, we obtain the conclusion.
\end{proof}

Next, we prove Theorem \ref{thm:mu and q-Hermite--Weber} which means that the image of the composition of the $q$-Borel and $q$-Laplace transformations of the divergent series $\widetilde{f}_{0}(x)$ is essentially equivalent to our $\mu (u,v;\alpha)$.

\begin{proof}[Proof of Theorem \ref{thm:mu and q-Hermite--Weber}]
From the definition of the $q$-Borel transformation and the $q$-binomial theorem (see \cite[p.~8, equation~(1.3.2)]{GR}),
\begin{align*}
\mathcal{B}^{+}\left(\widetilde{f}_{0}\right)(\xi )
 &=
 \sum_{n\geq 0}\frac{(a)_{n}}{(q)_{n}}(-1)^{n}\left(\frac{\xi }{a}\right)^{n}
 =
 \frac{(-\xi )_{\infty }}{(-\xi /a)_{\infty }}.
\end{align*}
By a simple calculation, we have
\begin{align*}
\mathcal{L}^{+}\circ\mathcal{B}^{+}\big(\widetilde{f}_{0}\big)(x,-\lambda)
 &=
 \sum_{n\in\Zz}\frac{1}{\theta_q(-\lambda q^n/x)}\frac{(\lambda q^n)_\infty}{(\lambda q^n/a)_\infty} \\
 &=
 \frac{1}{\theta_q(-\lambda/x)}\sum_{n\in\Zz}\left(-\frac{\lambda}{x}\right)^{n+1}q^\frac{n(n+1)}{2}
 \frac{\big(\lambda q^{n+1}\big)_{\infty }}{\big(\lambda q^{n-\alpha +1}\big)_{\infty }}.
\end{align*}
The second equality follows from the well-known $q$-difference relation of the theta function:
\[
\theta _{q}(q^{n}x)=q^{-\frac{n(n-1)}{2}}x^{-n}\theta _{q}(x), \qquad n \in \mathbb{Z}.
\]
By substituting $x={\rm e}^{2\pi{\rm i}(u-v)}$ and $\lambda={\rm e}^{2\pi{\rm i}u}$, we have
\begin{align*}
f_0\big({\rm e}^{2\pi{\rm i}(u-v)},-{\rm e}^{2\pi{\rm i}u}\big)
 &=
 {\rm e}^{\pi{\rm i}\alpha(u-v)}\mathcal{L}^{+}\circ\mathcal{B}^{+}\big(\widetilde{f}_{0}\big)\big({\rm e}^{2\pi{\rm i}(u-v)},-{\rm e}^{2\pi{\rm i}v}\big) \\
 &=
 -\frac{{\rm e}^{\pi{\rm i}\alpha(u-v)}}{\theta_q\big({-}{\rm e}^{2\pi{\rm i}v}\big)}\sum_{n\in\Zz}(-1)^n{\rm e}^{2\pi{\rm i}(n+1)v}q^\frac{n(n+1)}{2}\frac{\big({\rm e}^{2\pi{\rm i}u}q^{n+1}\big)_\infty}{\big({\rm e}^{2\pi{\rm i}u}q^{n-\alpha+1}\big)_\infty} \\
 &=
 \frac{{\rm i}{\rm e}^{\pi{\rm i}\alpha(u-v)}}{\vartheta_{11}(v)}\sum_{n\in\Zz}(-1)^n{\rm e}^{2\pi{\rm i}\left(n+\frac{1}{2}\right)v}q^\frac{\left(n+\frac{1}{2}\right)^2}{2}\frac{\big({\rm e}^{2\pi{\rm i}u}q^{n+1}\big)_\infty}{\big({\rm e}^{2\pi{\rm i}u}q^{n-\alpha+1}\big)_\infty} \\
 &=
 {\rm i}q^\frac{1}{8}\mu(u,v;\alpha).\tag*{\qed}
\end{align*}\renewcommand{\qed}{}
\end{proof}

Finally, we give a connection formula which derives the translation formula (\ref{eq:main results3}).
\begin{thm}
The following equation holds:
\label{thm:tuchimi connection}
\begin{gather}
f_{0}(x,\lambda^{\prime})
 =
 \frac{\theta_q(\lambda^{\prime})\theta_q(ax/\lambda^{\prime})x^\alpha}{\theta_q(\lambda^{\prime}/a)\theta_q(x/\lambda^{\prime})a}f_{\infty}(x,\lambda) \nonumber \\
\label{eq:tuchimi connection}
\hphantom{f_{0}(x,\lambda^{\prime})=}{}
 -
 (a)_\infty(q)_\infty^2\frac{\theta_q(-\lambda^{\prime}/x\lambda)\theta_q(-x)\theta_q(-\lambda\lambda^{\prime}/a)}{\theta_q(x\lambda)\theta_q(\lambda^{\prime}/x)\theta_q(\lambda^{\prime}/a)\theta_q(a/\lambda)}g_{\infty }(x).
\end{gather}
\end{thm}
\begin{proof}
We replace the variable $\lambda$ by $\lambda'$ in the connection formula~(\ref{amuconnect}) of the case of $f_0(x,\lambda)$:
\begin{gather}
\label{eq:different variable}
\begin{pmatrix}f_0(x,\lambda')\\ f_{\infty }(x,\lambda) \end{pmatrix}=-\frac{(q)_\infty}{(q/a)_\infty}\begin{pmatrix}\dfrac{\theta_q(\lambda')\theta_q(ax/\lambda')x^\alpha}{\theta_q(\lambda'/a)\theta_q(x/\lambda')a}&1\\ 1&\dfrac{\theta_q(\lambda)\theta_q(x\lambda/a)}{\theta_q(\lambda/a)\theta_q(x\lambda)}x^{-\alpha}\end{pmatrix}\begin{pmatrix}g_0(x)\\g_\infty(x)\end{pmatrix}.
\end{gather}
By erasing the function $g_0(x)$ in the above (\ref{eq:different variable}), we have
\begin{gather*}
f_0(x,\lambda^{\prime})-\frac{\theta_q(\lambda^{\prime})\theta_q(ax/\lambda^{\prime})x^\alpha}{\theta_q(\lambda^{\prime}/a)\theta_q(x/\lambda^{\prime})a}f_{\infty}(x,\lambda)\\
\qquad=\frac{(q)_\infty}{(q/a)_\infty}\left\{\frac{\theta_q(\lambda)\theta_q(\lambda^{\prime})\theta_q(x\lambda/a)\theta_q(ax/\lambda^{\prime})}{a\theta_q(\lambda/a)\theta_q(\lambda^{\prime}/a)\theta_q(x\lambda)\theta_q(x/\lambda^{\prime})}-1\right\}g_\infty(x)\\
\qquad =\frac{(a)_\infty(q)_\infty^2}{\theta_q(-a)}\frac{\theta_q(-\lambda^{\prime}/x\lambda)\theta_q(-x)}{\theta_q(x\lambda)\theta_q(\lambda^{\prime}/x)}C(x)g_\infty(x),
\end{gather*}
where
\[
C(x)=\frac{\lambda^{\prime}}{ax}\frac{\theta_q(\lambda)\theta_q(\lambda^{\prime})\theta_q(x\lambda/a)\theta_q(ax/\lambda^{\prime})} {\theta_q(-\lambda^{\prime}/x\lambda)\theta_q(-x)\theta_q(\lambda/a)\theta_q(\lambda^{\prime}/a)}-\frac{\theta_q(x\lambda) \theta_q(\lambda^{\prime}/x)}{\theta_q(-\lambda^{\prime}/x\lambda)\theta_q(-x)}.
\]
The function $C(x)$ has a simple pole in the points $x=q^j,\frac{\lambda^{\prime}}{\lambda}q^j$ for some $j\in\Zz$, and satisfies $C(xq)=C(x)$. We calculate the residue at the pole $x=1$,
\begin{align*}
\lim_{x\to1}\theta_q(-x)C(x)&=\frac{\lambda^{\prime}}{a}\frac{\theta_q(\lambda)\theta_q (\lambda^{\prime})\theta_q(\lambda/a)\theta_q(a/\lambda^{\prime})}{\theta_q(-\lambda^{\prime}/\lambda)\theta_q(\lambda/a)\theta_q(\lambda^{\prime}/a)} -\frac{\theta_q(\lambda)\theta_q(\lambda^{\prime})}{\theta_q(-\lambda^{\prime}/\lambda)}\\
&=\frac{\theta_q(\lambda)\theta_q(\lambda^{\prime})}{\theta_q(-\lambda^{\prime}/\lambda)}\left\{\frac{\lambda^{\prime}}{a} \frac{\theta_q(a/\lambda^{\prime})}{\theta_q(\lambda^{\prime}/a)}-1\right\}=0.
\end{align*}
It follows that the doubly periodic function $C(x)$ is constant in $x$.
Hence, we have
\begin{gather*}
C(x)=C(-a/\lambda)=-\frac{\theta_q(-a)\theta_q(-\lambda\lambda^{\prime}/a)}{\theta_q(\lambda^{\prime}/a)\theta_q(a/\lambda)}.\tag*{\qed}
\end{gather*}\renewcommand{\qed}{}
\end{proof}

\section{Proofs of the main results}\label{section3}
In this section, we prove the main theorems and corollaries.
First, we prove Theorem~\ref{thm:Thm1}.

\begin{proof}[Proof of Theorem \ref{thm:Thm1}]
The equation (\ref{eq:main results0}) is obvious from Theorem \ref{thm:mu and q-Hermite--Weber} and Lemma \ref{lem:lemma2}.

The pseudo periodicity (\ref{eq:main results1}) follows directly from the definition of $\mu(u,v;\alpha)$.

The forward shift (\ref{eq:main results2}) is proved as follows:
\begin{gather*}
\mu(u+\tau,v;\alpha)
 =
 \frac{{\rm e}^{\pi{\rm i}\alpha(u-v)}}{\vartheta_{11}(v)}q^\frac{\alpha}{2}\sum_{n\in\Zz}
 (-1)^n{\rm e}^{2\pi{\rm i}(n+\frac{1}{2})v}q^\frac{n(n+1)}{2} \\
\hphantom{\mu(u+\tau,v;\alpha)=}{}\times
 \frac{({\rm e}^{2\pi{\rm i}u}q^{n+2})_{\infty }}{({\rm e}^{2\pi{\rm i}u}q^{n-\alpha+2})_{\infty }}\big({\rm e}^{2\pi{\rm i}u}q^{n+1}+1-{\rm e}^{2\pi{\rm i}u}q^{n+1}\big) \\
\hphantom{\mu(u+\tau,v;\alpha)}{}=
 \frac{{\rm e}^{\pi{\rm i}\alpha(u-v)}}{\vartheta_{11}(v)}q^\frac{\alpha}{2}
 \left\{
 \sum_{n\in\Zz}{\rm e}^{2\pi{\rm i}(u-v)}(-1)^{n-1}{\rm e}^{2\pi{\rm i}(n+\frac{1}{2})v}q^\frac{n(n+1)}{2}
 \frac{({\rm e}^{2\pi{\rm i}u}q^{n+1})_{\infty }}{({\rm e}^{2\pi{\rm i}u}q^{n-\alpha+1})_{\infty }} \right. \\
 \left.
 \hphantom{\mu(u+\tau,v;\alpha)=\frac{{\rm e}^{\pi{\rm i}\alpha(u-v)}}{\vartheta_{11}(v)}q^\frac{\alpha}{2}}{}
 +
 \sum_{n\in\Zz}(-1)^n{\rm e}^{2\pi{\rm i}(n+\frac{1}{2})v}q^\frac{n(n+1)}{2}
 \frac{\big({\rm e}^{2\pi{\rm i}u}q^{n+1}\big)_{\infty }}{\big({\rm e}^{2\pi{\rm i}u}q^{n-(\alpha -1)+1}\big)_{\infty }}\right\} \\
 \hphantom{\mu(u+\tau,v;\alpha)}{}
 =
 -{\rm e}^{2\pi{\rm i}(u-v)}q^\frac{\alpha}{2}\mu(u,v;\alpha)+{\rm e}^{\pi{\rm i}(u-v)}q^\frac{\alpha}{2}\mu(u,v;\alpha-1).
\end{gather*}

The translation formula (\ref{eq:main results3}) is proved by putting $x={\rm e}^{2\pi{\rm i}(u-v)}$, $\lambda=-{\rm e}^{2\pi{\rm i}v}$, and $\lambda^{\prime}=-{\rm e}^{2\pi{\rm i}(u+z)}$ in the connection formula~(\ref{eq:tuchimi connection}) and using~(\ref{eq:mu and q-Hermite--Weber}).
The $\tau$-periodicity (\ref{eq:main results4}), symmetry (\ref{eq:main results4-2}) and pseudo-periodicity (\ref{eq:main results4-3}) are the case of $z=0$, $\tau$ and $-u-v+\alpha\tau$ in the translation formula (\ref{eq:main results3}), respectively,
\begin{gather*}
\mu(u,v;\alpha)
 =
 \frac{\vartheta_{11}(v-\alpha\tau)\vartheta_{11}(u)}{\vartheta_{11}(u-\alpha\tau)\vartheta_{11}(v)}{\rm e}^{2\pi{\rm i}\alpha(u-v)}\mu(v,u;\alpha), \\
\mu(u,v;\alpha)
 =
 \mu(u+\tau,v+\tau;\alpha), \\
\mu(u,v;\alpha)
 =
 \frac{\vartheta_{11}(v-\alpha\tau)\vartheta_{11}(u)}{\vartheta_{11}(u-\alpha\tau)\vartheta_{11}(v)}{\rm e}^{2\pi{\rm i}\alpha(u-v)}\mu(-u+\alpha\tau,-v+\alpha\tau;\alpha).
\end{gather*}

Finally to prove the equation (\ref{eq:main results5}), let us put
\[
\Phi(u,v;\alpha):=\frac{\vartheta_{11}(v-\alpha\tau)\vartheta_{11}(u)}{\vartheta_{11}(u-\alpha\tau)\vartheta_{11}(v)}{\rm e}^{2\pi{\rm i}\alpha(u-v)},
\]
which satisfies
\[
\Phi(u,v;\alpha)=\Phi(u,v;\alpha-1)=\Phi(u,v+\tau;\alpha)=\Phi(v,u;\alpha)^{-1}.
\]
Then, from (\ref{eq:main results3}) and (\ref{eq:main results2}), we have the desired result
\begin{align}
\mu(u,v+\tau;\alpha)
 &=
 \Phi(u,v+\tau;\alpha)\mu(v+\tau,u;\alpha) \nonumber \\
 &=
 \Phi(u,v;\alpha)\big({-}{\rm e}^{-2\pi{\rm i}(u-v)}q^\frac{\alpha}{2}\mu(v,u;\alpha)+{\rm e}^{-\pi{\rm i}(u-v)}q^\frac{\alpha}{2}\mu(v,u;\alpha-1)\big)\label{eq:v+tau} \\
 &=
 -{\rm e}^{-2\pi{\rm i}(u-v)}q^\frac{\alpha}{2}\mu(u,v;\alpha)+{\rm e}^{-\pi{\rm i}(u-v)}q^\frac{\alpha}{2}\mu(u,v;\alpha-1).\tag*{\qed}
\end{align}\renewcommand{\qed}{}
\end{proof}

We define the function $J_\nu^{(2)}$, one of Jackson $q$-Bessel functions, as follows \cite[p.~23]{KLS}:
\[
J_\nu^{(2)}(x;q):=\frac{\big(q^{\nu+1}\big)_\infty}{(q)_\infty}\left(\frac{x}{2}\right)^\nu{}_0\phi_1\hyper{-}{q^{\nu+1}}{q}{-\frac{x^2q^{\nu+1}}{4}}.
\]
It is known that this function has the following expression \cite[equation~(3.2)]{Ko}:
\begin{gather}\label{q-Bessel}
J_\nu^{(2)}(x;q)=\frac{1}{(q)_\infty}\left(\frac{x}{2}\right)^\nu{}_1\phi_1\hyper{-x^2/4}{0}{q}{q^{\nu+1}}.
\end{gather}
From the equations (\ref{eq:main results3}), (\ref{amuconnect}) and (\ref{q-Bessel}), we obtain some relations between the $q$-Bessel function and $\mu(u,v;\alpha)$.
\begin{cor}\label{rmk q-Bessel}
We have
\begin{align}\label{mu-j+j}
{\rm i}q^\frac{1}{8}\mu(u,v;\alpha)
 =
 \Phi(u,v;\alpha)j(u-v;\alpha)+j(v-u;\alpha),
\end{align}
and
\begin{gather*}
 \frac{\vartheta_{11}(\alpha\tau)\vartheta_{11}(u-v)\vartheta_{11}(z)\vartheta_{11}(u+v+z-\alpha\tau)}{\vartheta_{11}(u-\alpha\tau)\vartheta_{11}(v)\vartheta_{11}(u+z-\alpha\tau)\vartheta_{11}(v+z)}{\rm e}^{2\pi{\rm i}\alpha(u-v)}j(u-v;\alpha+1)\nonumber\\
 \qquad{}={\rm i}q^\frac{1}{8}\mu(u+z,v+z;\alpha+1)-{\rm i}q^\frac{1}{8}\mu(u,v;\alpha+1),
\end{gather*}
where
\begin{align*}
j(w;\alpha):={}& {\rm i}q^\frac{1}{8}\frac{(q)_\infty^2{\rm e}^{-\frac{\pi{\rm i}w}{2\tau}}}{\vartheta_{11}(w)(q^{1-\alpha})_\infty}J_{\frac{w}{\tau}}^{(2)}\big(2{\rm i}{\rm e}^{\pi{\rm i}(1-\alpha)\tau};q\big)\\
 ={}& {\rm i}q^\frac{1}{8}\frac{(q)_\infty}{(q^{1-\alpha})_\infty}\frac{{\rm e}^{\pi{\rm i}(1-\alpha)w}}{\vartheta_{11}(w)}\,{}_1\phi_1\hyper{q^{1-\alpha}}{0}{q}{{\rm e}^{2\pi{\rm i}w}q}.
\end{align*}
\end{cor}

\begin{rmk}
If $\alpha$ is a positive integer $k$ in the equation (\ref{mu-j+j}), the function $j(w;\alpha)$ diverges.
So, taking the limit $\alpha \to k$ of (\ref{mu-j+j}) is not straightforward.
For example, if we take the limit $\alpha \to 1$ in an appropriate setting, we obtain
\begin{gather}
\label{eq: Bmconect}
\vartheta_{11}(u-v)\mu(u,v)
=\frac{1}{2\pi{\rm i}}\left(\frac{\vartheta_{11}'(u)}{\vartheta_{11}(u)}-\frac{\vartheta_{11}'(v)}{\vartheta_{11}(v)}\right)+\sum_{n \in \mathbb{Z}\setminus \{0\}}\frac{(-1)^nq^\frac{n(n+1)}{2}}{1-q^n}{\rm e}^{2\pi{\rm i}n(u-v)}.
\end{gather}
This equation (\ref{eq: Bmconect}) is also obtained by the specialization of (\ref{eq:mu translation})
\[
\mu(u,v+\epsilon)=\mu(\epsilon,u-v)+\frac{{\rm i}q^{\frac{1}{8}}(q)_{\infty }^{3}\vartheta_{11}(v)\vartheta_{11}(u+\epsilon)}{\vartheta_{11}(u)\vartheta_{11}(v+\epsilon)\vartheta_{11}(\epsilon)\vartheta_{11}(u-v)},
\]
and taking the limit $\epsilon \to 0$.
\end{rmk}
Next, we prove Corollary \ref{cor:Thm1 k}.
Since all other equations except equation (\ref{eq:mu k 7}) are obtained immediately from Theorem~\ref{thm:Thm1}, we prove only~(\ref{eq:mu k 7}).

\begin{proof}[Proof of Corollary \ref{cor:Thm1 k}]
The first equation is obvious by the Definition \ref{def:mua}.
Next, using the partial fraction decomposition
\[
\prod_{l=0}^{k-1}\frac{1}{1-{\rm e}^{2\pi{\rm i}u}q^{n-l}}
 =
 \sum_{j=0}^{k-1}\frac{(-1)^{k-j-1}}{(q)_j(q)_{k-1-j}}\frac{q^{\frac{(k-j)(k-j-1)}{2}}}{1-{\rm e}^{2\pi{\rm i}u}q^{n-j}},
\]
we have
\begin{align*}
\mu(u,v;k)
 &=
 \frac{{\rm e}^{\pi{\rm i}k(u-v)}}{\vartheta_{11}(v)}\sum_{j=0}^{k-1}\sum_{n\in\Zz}\frac{(-1)^n{\rm e}^{2\pi{\rm i}(n+\frac{1}{2})v}q^\frac{n(n+1)}{2}}{1-{\rm e}^{2\pi{\rm i}u}q^{n-j}}\frac{(-1)^{k-1-j}q^\frac{(k-j)(k-j-1)}{2}}{(q)_j(q)_{k-1-j}} \\
 &=
 {\rm e}^{\pi{\rm i}(k-1)(u-v+\tau)}\sum_{j=0}^{k-1}\frac{(-1)^{k-1-j}}{(q)_j(q)_{k-1-j}}q^\frac{(k-1-j)^2}{2}\mu(u-j\tau,v).\tag*{\qed}
\end{align*}
\renewcommand{\qed}{}
\end{proof}

As an important application of the recursion formula (\ref{eq:main results5}), we also obtain the following relation (reducing formula) for $\mu(u,v;\alpha)$.
\begin{cor}
\label{cor:mu a reduction}
For $\alpha\in\Cz$, we have
\begin{gather}
\label{eq:mu a reduction}
\mu\left(u,u+\frac{1}{2};\alpha-1\right)=\big(q^{-\alpha}-1\big)\mu\left(u,u+\frac{1}{2};\alpha+1\right).
\end{gather}
In particular, for any non-negative integer $k$,
\begin{gather*}
\mu\left(u,u+\frac{1}{2};2k\right)
 =
 -{\rm i}q^{-\frac{1}{8}}\frac{q^{k^2}}{\big(q;q^2\big)_{k}}, \\
\mu\left(u,u+\frac{1}{2};2k+1\right)
 =
 \frac{q^{k(k+1)}}{\big(q^2;q^2\big)_{k}}\mu\left(u,u+\frac{1}{2}\right)
 =
 \frac{q^{k(k+1)}}{\big(q^2;q^2\big)_{k}}\frac{{\rm i}q^{\frac{1}{4}}(q)_{\infty }^{4}(-q)_{\infty }^{2}\vartheta_{11}(2u)}{\vartheta_{11}(u)^{2}\vartheta_{11}\big(u+\frac{1}{2}\big)^{2}}.
\end{gather*}
\end{cor}

\begin{proof}[Proof of Corollary \ref{cor:mu a reduction}]
It is enough to show that
\begin{gather*}
\mu\left(u,u+\frac{1}{2}\right)
 =
 \frac{{\rm i}q^{\frac{1}{4}}(q)_{\infty }^{4}(-q)_{\infty }^{2}\vartheta_{11}(2u)}{\vartheta_{11}(u)^{2}\vartheta_{11}\left(u+\frac{1}{2}\right)^{2}}.
\end{gather*}
If we put $v=u-\frac{1}{2}$ and $z=\frac{1}{2}$ in (\ref{eq:mu translation}), then
\begin{gather*}
\mu\left(u,u+\frac{1}{2}\right)
 =
 \mu\left(u-\frac{1}{2},u\right)
 +\frac{{\rm i}q^{\frac{1}{8}}(q)_{\infty }^{3}\vartheta_{11}\big(\frac{1}{2}\big)\vartheta_{11}(2u)}{\vartheta_{11}\big(u-\frac{1}{2}\big)\vartheta_{11}(u)^{2}\vartheta_{11}\big(u+\frac{1}{2}\big)}.
\end{gather*}
Since
\begin{gather*}
\mu\left(u-\frac{1}{2},u\right)
 =
 -\mu\left(u+\frac{1}{2},u\right)
 =-\mu\left(u,u+\frac{1}{2}\right), \\
\vartheta_{11}\left(u-\frac{1}{2}\right)
 =
 -\vartheta_{11}\left(u+\frac{1}{2}\right),
\end{gather*}
and
\[
\vartheta_{11}\left(\frac{1}{2}\right)
 =
 -2q^{\frac{1}{8}}(q,-q,-q)_{\infty },
\]
we obtain the conclusion.
\end{proof}

\begin{rmk}
(1) Corollary \ref{cor:mu a reduction} is a generalization of a classical evaluation~\cite{G}.
In fact, Gauss proved the case of negative integers in our formula~(\ref{eq:mu a reduction}), which is the special value of continuous $q$-Hermite polynomial $H_{n}(x \mid q)$ at $x=0$:
\begin{gather}
\label{eq:Gauss evaluation}
{\rm i}^{-n}H_{n}(0 \mid q)
 =
 \frac{\big(q^{n-1};q^{-2}\big)_{\infty}}{\big(q^{-1};q^{-2}\big)_{\infty}}
 =\begin{cases}
 0 & \text{if } n=2N-1,\\
 (q;q^{2})_{N} & \text{if } n=2N.
 \end{cases}
\end{gather}
This formula (\ref{eq:Gauss evaluation}) is a key lemma to derive the product formula of the quadratic Gauss sum:
\begin{gather}
\label{eq:Gauss sum product}
\sum_{k=1}^{2N}{\rm e}^{\frac{2\pi{\rm i} k^{2}}{2N+1}}
 =
 \prod_{j=1}^{N}
 \big({\rm e}^{\frac{2\pi{\rm i}}{2N+1}(2j-1)}-{\rm e}^{-\frac{2\pi{\rm i}}{2N+1}(2j-1)}\big).
\end{gather}
In fact, by using this product formula (\ref{eq:Gauss sum product}), Gauss determined the sign of the quadratic Gauss sum, and gave a proof (the 4th proof) of the quadratic reciprocity law.

(2) The function $F(u):=\mu \big(u,u+\frac{1}{2}\big)$ also has the following expressions:
\begin{align}
\label{eq:mu and elliptic}
F(u)
 &=
 -\frac{1}{2\pi{\rm i}}\frac{2K}{\vartheta _{11}\big(\frac{1}{2}\big)}\frac{1}{\sn{(2Ku,k)}\sn{\big(2K\big(u+\frac{1}{2}\big),k\big)}} \\
\label{eq:mu and elliptic2}
 &=
 -\frac{1}{2\pi{\rm i}}\frac{2K}{\vartheta _{11}\big(\frac{1}{2}\big)}\frac{\dn{(2Ku,k)}}{\sn{(2Ku,k)}\cn{(2Ku,k)}}.
\end{align}
Here, $k=k(\tau )$ is the elliptic modulus
and $K=K(\tau )$ is the elliptic period:
\[
K:=\frac{\pi }{2}(q)_{\infty }^{2}\big({-}q^{\frac{1}{2}}\big)_{\infty }^{4}=\int_{0}^{1}\frac{{\rm d}t}{\sqrt{\big(1-t^{2}\big)\big(1-k^{2}t^{2}\big)}},
\]
and $\mathrm{sn}$, $\mathrm{cn}$ and $\mathrm{dn}$ are Jacobian elliptic functions (see, for example, \cite[Chapter~11]{BW}).

Although this fact is an exercise of the $\mu$-function and elliptic functions, we have not been able to find any appropriate references about this result.
Hence, we give a sketch of the proof of (\ref{eq:mu and elliptic}) and (\ref{eq:mu and elliptic2}).

From the equations (\ref{eq:mu pseudo periodicity}) and (\ref{eq:mu symmetry3}), the function $F(u)$ is odd:
\begin{gather*}
F(-u) =
 \mu \left(-u,-u+\frac{1}{2}\right)
 =
 \mu \left(u,u-\frac{1}{2}\right)
 =
 -\mu \left(u,u+\frac{1}{2}\right)
 =
 -F(u).
\end{gather*}
The double periodicity of $F(u)$ follows from (\ref{eq:mu pseudo periodicity}) and (\ref{eq:mu symmetry}):
\begin{gather*}
F(u+1)
 =
 \mu \left(u+1,u+1+\frac{1}{2}\right)
 =
 (-1)^{2}\mu \left(u,u+\frac{1}{2}\right)
 =
 F(u), \\
F(u+\tau )
 =
 \mu \left(u+\tau,u+\tau +\frac{1}{2}\right)
 =
 \mu \left(u,u+\frac{1}{2}\right)
 =
 F(u).
\end{gather*}
By the definition of Zwegers' $\mu$-function, the function $F(u)$ has simple poles at each lattice point on $\mathbb{Z}+\mathbb{Z}\tau $ and $\mathbb{Z}+\frac{1}{2}+\mathbb{Z}\tau $ and no other singularities, and the residues of $F(u)$ at each pole are the following:
\begin{gather*}
\mathop{\mathrm{Res}}_{u=0}F(u)
 =
 \lim_{u \to 0}\frac{{\rm e}^{\pi{\rm i}u}}{\vartheta _{11}\big(u+\frac{1}{2}\big)}\frac{u}{1-{\rm e}^{2\pi{\rm i}u}}
 =
 -\frac{1}{2\pi{\rm i}}\frac{1}{\vartheta _{11}\big(\frac{1}{2}\big)}, \\
\mathop{\mathrm{Res}}_{u=-\frac{1}{2}}F(u)
 =
 \lim_{u \to -\frac{1}{2}}\frac{{\rm e}^{\pi{\rm i}\left(u+\frac{1}{2}\right)}}{\vartheta _{11}(u)}\frac{u}{1-{\rm e}^{2\pi{\rm i}\big(u+\frac{1}{2}\big)}}
 =
 \frac{1}{2\pi{\rm i}}\frac{1}{\vartheta _{11}\big(\frac{1}{2}\big)}.
\end{gather*}

On the other hand, by a simple calculation of elliptic functions, the right-hand side of (\ref{eq:mu and elliptic}) has exactly the same properties as $F(u)$.
Since the function
\[
G(u):=F(u)+\frac{1}{2\pi{\rm i}}\frac{2K}{\vartheta _{11}(\frac{1}{2})}\frac{1}{\sn{(2Ku,k)}\sn{\big(2K\big(u+\frac{1}{2}\big),k\big)}}
\]
is odd and a constant with respect to $u$, $G(u)$ is identically zero.

Finally, by (\ref{eq:mu and elliptic}) and the translation formula
\[
\sn(z+K,k)=\frac{\cn(z,k)}{\dn(z,k)},
\]
we obtain the equation (\ref{eq:mu and elliptic2}).
\end{rmk}

Next, we prove Theorem \ref{thm:mu and CqH} for the relation between our $\mu$-functions $\mu(u,v;k)$ and the continuous $q$-Hermite polynomials.

\begin{proof}[Proof of Theorem \ref{thm:mu and CqH}]
For $n=0,1,2,\dots$, the continuous $q$-Hermite polynomials satisfy the following recursion equation (see \cite[p.~541, equation~(14.26.3)]{KLS}):
\begin{gather}\label{q-Her Rec}
2xH_n(x\mid q)=H_{n+1}(x\mid q)+\big(1-q^n\big)H_{n-1}(x\mid q), \\
\label{q-Her initial}
H_{0}(x\mid q)=1, \qquad H_1(x\mid q)=2x.
\end{gather}
Let us now set
\[
\widehat{H}_{k}(x\mid q):={\rm i}q^\frac{1}{8}\mu(u,v;-k), \qquad x=\cos\pi(u-v).
\]
To prove $\widehat{H}_{k}(x\mid q)={H}_{k}(x\mid q)$, it is enough to show that the $\widehat{H}_{k}(x\mid q)$ satisfy the recurrence relation (\ref{q-Her Rec}) and the initial condition (\ref{q-Her initial}).
The recurrence relation (\ref{q-Her Rec}) is equal to the formula (\ref{eq:mu k 5}) exactly.
For the initial condition (\ref{q-Her initial}), by the definition and the triple product identity, we have
\begin{gather}\label{eq:Hhat initial0}
\widehat{H}_{0}(x\mid q)
 =
 {\rm i}q^\frac{1}{8}\mu(u,v;0)
 =
 \frac{1}{\vartheta_{11}(v)}
 \sum_{n\in\Zz}{\rm e}^{2\pi{\rm i}(n+\frac{1}{2})(v+\frac{1}{2})+\pi{\rm i}(n+\frac{1}{2})^{2}\tau }
 =
 1.
\end{gather}
The equation $\widehat{H}_{1}(x\mid q)=2x$ follows from (\ref{eq:Hhat initial0}) and the case of $k=0$ in (\ref{eq:mu k 5}).
\end{proof}

\begin{rmk}
In the equation (\ref{mu-j+j}), let $\alpha$ be a negative integer $-n$.
In this case, the equation~(\ref{mu-j+j}) is
\begin{gather}
H_n(x\mid q)=\frac{{\rm i}q^\frac{1}{8}(q)_n}{\vartheta_{11}(u-v)}\left\{{\rm e}^{\pi{\rm i}(1+n)(u-v)}{}_1\phi_1\hyper{q^{1+n}}{0}{q}{{\rm e}^{2\pi{\rm i}(u-v)}q}\right.\nonumber\\
\label{eq:Hermite 2 sum expression}
\left.
\hphantom{H_n(x\mid q)=\frac{{\rm i}q^\frac{1}{8}(q)_n}{\vartheta_{11}(u-v)}}{}
-{\rm e}^{-\pi{\rm i}(1+n)(u-v)}{}_1\phi_1\hyper{q^{1+n}}{0}{q}{{\rm e}^{-2\pi{\rm i}(u-v)}q}\right\}.
\end{gather}
This is a non-trivial result that the right-hand side of (\ref{eq:Hermite 2 sum expression}) is the sum of two infinite sums, but cancels out with the theta function to become a finite sum.

This equation is also regarded as the $a,b,c,d \to 0$ limit case of an Askey--Wilson polynomial's expression by the sum of two ${}_4\phi_3$ functions \cite[(2.11)]{S} (but this reference has small typos. The denominator terms of ${}_4\phi_3$ (${}_4\varphi_3$) $q^{1-\nu+z}$ and $q^{1-\nu-z}$ should be $q^{1-\nu+z}/d$ and $q^{1-\nu-z}/d$, respectively).
\end{rmk}

Next, we prove Theorem \ref{thm:Sr} for the generating function $S(r)$.

\begin{proof}[Proof of Theorem \ref{thm:Sr}]
(1) Since the equation (\ref{eq:Sr rec2}) follows from (\ref{eq:Sr rec1}) immediately, it is enough to show (\ref{eq:Sr rec1}).
We find by (\ref{eq:mu k 5}) that
\begin{gather*}
S(r)-S\left(\frac{r}{q}\right)=
 \sum_{k=0}^\infty\big(1-q^{-k}\big)\mu(u,v;k+1)r^k \\
\hphantom{S(r)-S\left(\frac{r}{q}\right)}{}
=
 \sum_{k=0}^\infty(2\cos\pi(u-v)\mu(u,v;k)-\mu(u,v;k-1))r^k \nonumber\\
\hphantom{S(r)-S\left(\frac{r}{q}\right)}{}
=
 2r\cos\pi(u-v)\sum_{k=-1}^\infty\mu(u,v;k+1)r^k-r^2\sum_{k=-2}^\infty\mu(u,v;k+1)r^k \nonumber\\
\hphantom{S(r)-S\left(\frac{r}{q}\right)}{}=
 \big(2\cos\pi(u-v)-r^2\big)S(r)-r\mu(u,v;0) \\
\hphantom{S(r)-S\left(\frac{r}{q}\right)=}{}
 +2\cos\pi(u-v)\mu(u,v;0)-\mu(u,v;-1). \nonumber
\end{gather*}
By the case of $k=0$ in recursion equation (\ref{eq:mu k 5}):
\begin{gather*}
2\cos\pi(u-v)\mu(u,v;0)-\mu(u,v;-1)=0
\end{gather*}
and $\mu(u,v;0)=-{\rm i}q^{-\frac{1}{8}}$, we have the conclusion.

(2) Let $N=0,1,2,\dots$. Using the equation (\ref{eq:mu k 5}), we find that
\begin{align}
S(r)
 &=
 \big(1-r{\rm e}^{\pi{\rm i}(u-v)}q\big)\big(1-r{\rm e}^{-\pi{\rm i}(u-v)q}\big)S(rq)-{\rm i}rq^\frac{7}{8} \nonumber \\
 &=
\label{eq:N expression Sr}
 \big(r{\rm e}^{\pi{\rm i}(u-v)}q,r{\rm e}^{-\pi{\rm i}(u-v)}q\big)_NS\big(rq^N\big)-
 {\rm i}rq^\frac{7}{8}\sum_{n=0}^{N-1}q^n\big(r{\rm e}^{\pi{\rm i}(u-v)}q,r{\rm e}^{-\pi{\rm i}(u-v)}q\big)_n.
\end{align}
Taking the limit $N \to \infty$ in (\ref{eq:N expression Sr}), we have
\begin{align}
S(r)
 &=
 \big(r{\rm e}^{\pi{\rm i}(u-v)}q,r{\rm e}^{-\pi{\rm i}(u-v)}q\big)_\infty\mu(u,v)-
 {\rm i}rq^\frac{7}{8}\sum_{n=0}^\infty q^n\big(r{\rm e}^{\pi{\rm i}(u-v)}q,r{\rm e}^{-\pi{\rm i}(u-v)}q\big)_n \nonumber \\
 &=
 \big(r{\rm e}^{\pi{\rm i}(u-v)}q,r{\rm e}^{-\pi{\rm i}(u-v)}q\big)_{\infty }
 \mu(u,v)
 -
 {\rm i}rq^{\frac{7}{8}}\,{_{3}\phi _2}\left(\begin{matrix}q, r{\rm e}^{\pi{\rm i}(u-v)}q, r{\rm e}^{-\pi{\rm i}(u-v)}q \\ 0,0 \end{matrix};q,q\right).\!\!\!\label{eq:gen Sr expression}
\end{align}
The conclusion follows from (\ref{eq:gen Sr expression}) and Andrews' formula \cite[p.~298, Exercise 10.8]{GR}:
\begin{gather}
\label{eq:Andrews formula}
{_{3}\phi _2}\left(\begin{matrix}a,b,c \\ d,e \end{matrix};q,x\right)
 =
 \frac{(ax,b,c)_{\infty }}{(x,d,e)_{\infty }}
 \Phi^{(1)}\left(\begin{matrix}x;d/b,e/c \\ ax \end{matrix};q;b,c\right).
\end{gather}

Finally, we prove (\ref{eq:Sr expression3}).
From (\ref{eq:Sr expression2}), we have
\begin{gather*}
S(r) =
 \big(r{\rm e}^{\pi{\rm i}(u-v)}q,r{\rm e}^{-\pi{\rm i}(u-v)}q\big)_{\infty }\\
\hphantom{S(r) =}{}\times \left\{
 \mu(u,v)-\frac{{\rm i}rq^{\frac{7}{8}}}{1-q}\Phi ^{(1)}\left(\begin{matrix}q;0,0 \\ q^2 \end{matrix};q;r{\rm e}^{\pi{\rm i}(u-v)}q,r{\rm e}^{-\pi{\rm i}(u-v)}q\right)\right\} \\
\hphantom{S(r)}{}
=
 \big(r{\rm e}^{\pi{\rm i}(u-v)}q, r{\rm e}^{-\pi{\rm i}(u-v)}q\big)_{\infty } \\
\hphantom{S(r) =}{}\times
 \left\{
 \mu(u,v)
 -\frac{{\rm i}q^{\frac{7}{8}}}{1-q}\sum_{n\geq 0}\sum_{k=0}^{n}
 \frac{(q)_{n}}{(q)_{k}(q)_{n-k}(q^{2})_{n}}{\rm e}^{\pi{\rm i}(2k-n)(u-v)}q^{n}r^{n+1}\right\} \\
\hphantom{S(r)}{} =
 \big(r{\rm e}^{\pi{\rm i}(u-v)}q, r{\rm e}^{-\pi{\rm i}(u-v)}q\big)_{\infty }
\left\{
 \mu(u,v)
 -\sum_{n\geq 0}\frac{{\rm i}q^{\frac{7}{8}}H_{n}(\cos{\pi (u-v)}\mid q)}{(q)_{n+1}}q^{n}r^{n+1}\right\}.
\end{gather*}
From Theorem \ref{thm:mu and CqH}, $H_{n}(x\mid q):={\rm i}q^\frac{1}{8}\mu(u,v;-n)$, so (\ref{eq:Sr expression3}) holds.
\end{proof}

By re-expanding the generating function $S(r)$, we obtain the following relationship between the function $\mu(u,v;k+1)$ and the continuous $q$-Hermite polynomials.
\begin{cor}\label{cor:Sr expressions}
For non-negative integers $k$, we have
\begin{gather}
\label{eq:Sr expressions}
\mu(u,v;k+1)
 =
 \sum_{l=0}^{k}
 \frac{q^{l}}{(q)_{l}}F_{k-l+1}(\cos\pi(u-v)\mid q)\mu (u,v;1-l),
\end{gather}
where
\begin{align*}
F_{n+1}(\cos\pi w\mid q)
 :={}&
 {\rm e}^{- \pi{\rm i}nw}\, {}_1\phi_1\hyper{q^{-n}}{0}{q}{{\rm e}^{2 \pi{\rm i}w}q}\frac{(-1)^nq^\frac{n(n+1)}{2}}{(q)_n}\\
={}&
 {\rm e}^{-\frac{\pi{\rm i}w}{2\tau}}\frac{(q)_\infty}{(q^{-n})_n} J_\frac{w}{\tau}^{(2)}\big(2{\rm i}{\rm e}^{-\pi{\rm i}n\tau};q\big).
\end{align*}
For positive integers $m$, we have
\begin{gather}
\label{eq:Sr expressions 2}
\sum_{k=0}^{m}\mu{(u,v;k+1)}\frac{H_{m-k}(\cos{\pi (u-v)}\mid q)}{(q)_{m-k}}q^{m-k}
 =
 -\frac{{\rm i}q^\frac{7}{8}H_{m-1}(\cos{\pi (u-v)}\mid q)}{(q)_{m}}.
\end{gather}
\end{cor}
\begin{proof}
From the Taylor expansion of the $q$-exponential function \cite[equation~(II.2)]{GR}
\[
(x)_{\infty }=\sum_{n\geq 0}\frac{1}{(q)_{n}}(-1)^nq^\frac{n(n-1)}{2}x^{n},
\]
we have
\begin{align*}
\big({\rm e}^{{\rm i}\theta }rq,{\rm e}^{-{\rm i}\theta }rq\big)_{\infty }
 &=
 \sum_{n\geq 0}(-qr)^{n}\sum_{k=0}^{n}\frac{q^{\binom{k}{2}}q^{\binom{n-k}{2}}}{(q )_{k}(q )_{n-k}}{\rm e}^{{\rm i}(n-2k)\theta }\\
 &=
 \sum_{n\geq 0}{}_1\phi_1\hyper{q^{-n}}{0}{q}{{\rm e}^{2{\rm i}\theta}q}\frac{(-1)^nq^\frac{n(n-1)}{2}}{(q)_n}({\rm e}^{-{\rm i}\theta}rq)^n\\
 &=
 \sum_{n\geq0}F_{n+1}(\cos\pi(u-v)\mid q)r^n.
\end{align*}
By expanding (\ref{eq:Sr expression3}) and comparing the coefficient in (\ref{eq:Sr expression3}), we obtain (\ref{eq:Sr expressions}).
To prove (\ref{eq:Sr expressions 2}), we divide by both sides of (\ref{eq:Sr expression3}) by $\big({\rm e}^{{\rm i}\theta }rq,{\rm e}^{-{\rm i}\theta }rq\big)_{\infty }$;
\begin{align}
\label{eq:Sr expression3 divid}
\frac{1}{(r{\rm e}^{\pi{\rm i}(u-v)}q,r{\rm e}^{-\pi{\rm i}(u-v)}q)_{\infty }}S(r)
 =
 \sum_{m\geq 0}\frac{\mu(u,v;1-m )}{(q)_{m}}q^{m}r^{m}.
\end{align}
By the generating function of continuous $q$-Hermite polynomial (\ref{eq:gen func of CqH}), the left-hand side of~(\ref{eq:Sr expression3 divid}) is equal to
\[
\sum_{m\geq 0}\sum_{k=0}^{m}\mu{(u,v;k+1)}\frac{H_{m-k}(\cos{\pi (u-v)}\mid q)}{(q)_{m-k}}q^{m-k}r^{m}.
\]
Then we obtain the conclusion (\ref{eq:Sr expressions 2}).
\end{proof}

\section[Modular transformations related to mu(u,v;k,tau)]{Modular transformations related to $\boldsymbol{\mu(u,v;k,\tau)}$}\label{section4}
In this section, again let $k$ be a positive integer.
We state that $\mu(u,v;k,\tau)$ is essentially equivalent to the original $\mu$-function with respect to the properties of a real-analytic Jacobi form.
\begin{prop}
We define a modified $\mu(u,v;k,\tau)$ by
\begin{gather*}
\tilde{\nu} (u,v;k,\tau)
 :=
 \frac{\mu(u,v;k+1,\tau)}{F_{k+1}(\cos\pi(u-v)\mid q)}+\frac{1}{2{\rm i}}R_{k+1}(u-v;\tau),
\end{gather*}
where
\[
R_{k+1}(u;\tau):=R(u;\tau)-2q^{-\frac{1}{8}}\sum_{l=1}^k\frac{q^l}{(q)_l}\frac{F_{k-l+1}(\cos\pi u\mid q)}{F_{k+1}(\cos\pi u\mid q)}H_{l-1}(\cos\pi u \mid q).
\]
We have
\begin{gather*}
\tilde{\mu}(u,v;\tau)=\tilde{\nu} (u,v;k,\tau).
\end{gather*}
Therefore, the following transformations hold;
\begin{gather}
\label{eq:tnu +1}
\tilde{\nu} (u,v;k,\tau+1)
 =
 {\rm e}^{-\frac{\pi{\rm i}}{4}}\tilde{\nu}(u,v;k,\tau ), \\
\label{eq:tnu +tau}
\tilde{\nu} \left(\frac{u}{\tau},\frac{v}{\tau};k,-\frac{1}{\tau}\right)
 =
 -{\rm i}\sqrt{-{\rm i}\tau}{\rm e}^{\pi{\rm i}\frac{(u-v)^2}{\tau}}\tilde{\nu} (u,v;k,\tau ).
\end{gather}
\end{prop}
\begin{proof}
From equation (\ref{eq:Sr expressions}) of Corollary \ref{cor:Sr expressions},
\begin{gather*}
 \mu(u,v;\tau)-\frac{\mu(u,v;k+1,\tau)}{F_{k+1}(\cos\pi(u-v)\mid q)} \\
\qquad{} =
 {\rm i}q^{-\frac{1}{8}}\sum_{l=1}^k\frac{q^l}{(q)_l}\frac{F_{k-l+1}(\cos\pi(u-v)\mid q)}{F_{k+1}(\cos\pi(u-v)\mid q)}H_{l-1}(\cos\pi(u-v)\mid q). 
\end{gather*}
Therefore,
\begin{gather*}
\tilde{\mu}(u,v;\tau)
 =
 \mu(u,v;\tau)+\frac{1}{2{\rm i}}R(u-v;\tau) \\
\hphantom{\tilde{\mu}(u,v;\tau)}{}
=\frac{\mu(u,v;k+1,\tau)}{F_{k+1}(\cos\pi(u-v)\mid q)}+\frac{1}{2{\rm i}}R(u-v;\tau) \\
\hphantom{\tilde{\mu}(u,v;\tau)=}{}
 +{\rm i}q^{-\frac{1}{8}}\sum_{l=1}^k\frac{q^l}{(q)_l}\frac{F_{k-l+1}(\cos\pi(u-v)\mid q)}{F_{k+1}(\cos\pi(u-v)\mid q)}H_{l-1}(\cos\pi(u-v)\mid q) \\
\hphantom{\tilde{\mu}(u,v;\tau)}{}
=
 \frac{\mu(u,v;k+1,\tau)}{F_{k+1}(\cos\pi(u-v)\mid q)}+\frac{1}{2{\rm i}}R_{k+1}(u-v;\tau).\tag*{\qed}
\end{gather*}\renewcommand{\qed}{}
\end{proof}

\begin{rmk}T.~Matsusaka also mentions this proposition~\cite{M}.
\end{rmk}

\section{Remarks for further studies}\label{section5}

In this paper, from the point of view of the analysis of one-variable linear $q$-difference equations of the Laplace type, we have studied a generalization of Zwegers' $\mu$-function $\mu(x,y;a)$ which satisfies the $q$-Hermite--Weber equation (\ref{eq:q-Hermite}).

On the other hand, the generalized $\mu$-function $\mu(x,y;a)$ is a two-variable function originally.
More precisely, $\mu(x,y;a)$ is closely related to the $q$-Appell hypergeometric function $\Phi^{(1)}$ (\ref{q-Appell}) and its $q$-difference system which we call $q$-Appell difference system:
\begin{gather}
[(1-T_{x})(1-c/q T_{x}T_{y})-x(1-aT_{x}T_{y})(1-b_{1}T_{x})]\Phi (x,y)=0,\nonumber \\
[(1-T_{y})(1-c/q T_{x}T_{y})-y(1-aT_{x}T_{y})(1-b_{2}T_{y})]\Phi (x,y)=0, \nonumber\\
\label{eq:Asystem3}
[x(1-T_{y})(1-b_{1}T_{x})-y(1-T_{x})(1-b_{2}T_{y})]\Phi (x,y)=0.
\end{gather}
The first and second equations are essentially equivalent to Gasper--Rahman \cite[Exer\-ci\-ses~10.12(i) and~(ii)]{GR} which have some small typos (the terms $(c/q-a)f(qx,qy)$ in (i) and (ii) should be $(c/q-ax)f(qx,qy)$ and $(c/q-ay)f(qx,qy)$).
The third equation (\ref{eq:Asystem3}) is a reducible factor of the following $q$-difference equation:
\begin{gather*}
 (1-T_{x})[(1-T_{y})(1-c/q T_{x}T_{y})-y(1-aT_{x}T_{y})(1-b_{2}T_{y})]\Phi (x,y) \nonumber \\
\qquad \quad {}-
 (1-T_{y})[(1-T_{x})(1-c/q T_{x}T_{y})-x(1-aT_{x}T_{y})(1-b_{1}T_{x})]\Phi (x,y) \nonumber \\
\qquad{} =
 (1-aT_{x}T_{y})[x(1-T_{y})(1-b_{1}T_{x})-y(1-T_{x})(1-b_{2}T_{y})]\Phi (x,y)
 =0
\end{gather*}
and we easily show that $\Phi^{(1)}$ satisfies~(\ref{eq:Asystem3}).

First, $\Phi^{(1)}$ appears in the expression of $\mu(x,y;a)$.
In Theorem~\ref{Theorem 3}, by dividing the bilateral sum of $q$-hypergeometric series ${}_2\psi_2$ and ${}_0\psi_2$ into two parts, positive and negative;
\begin{gather*}
{}_2\psi_2\left(\begin{matrix}x/a,y/a\\0,0\end{matrix};q,a\right)
 =
 \frac{xy}{a}\left(1-\frac{a}{x}\right)\left(1-\frac{a}{y}\right){}_3\phi_2\left(\begin{matrix}xq/a,yq/a,q \\ 0,0 \end{matrix};q,a\right) \\
\hphantom{{}_2\psi_2\left(\begin{matrix}x/a,y/a\\0,0\end{matrix};q,a\right)=}{}
 +{}_2\phi_2\left(\begin{matrix}0,q \\ aq/x,aq/y \end{matrix};q,-\frac{aq^2}{xy}\right), \\
{}_0\psi_2\left(\begin{matrix}-\\x,y\end{matrix};q,\frac{xy}{a}\right)
 =
 \frac{1}{(1-x)(1-y)}{}_1\phi_2\left(\begin{matrix}q \\ xq,yq \end{matrix};q,\frac{xyq^2}{a}\right)
 +{}_3\phi_2\left(\begin{matrix}q/x,q/y,q \\ 0,0 \end{matrix};q,a\right),
\end{gather*}
we rewrite Theorem \ref{Theorem 3} as follows:
\begin{gather*}
{\rm i}q^{\frac{1}{8}}\mu(x,y;a)
 =
 {\rm e}^{\pi{\rm i}\alpha(u-v)}\left\{\frac{xy}{a}\frac{(a,q,a/x,a/y)_\infty}{\theta_q(-y)\theta_q(-x/a)}
 {}_3\phi_2\left(\begin{matrix}xq/a,yq/a,q \\ 0,0 \end{matrix};q,a\right) \right. \nonumber\\
\left.
\hphantom{{\rm i}q^{\frac{1}{8}}\mu(x,y;a) = {\rm e}^{\pi{\rm i}\alpha(u-v)}}{}
 +\frac{(a,q,aq/x,aq/y)_\infty}{\theta_q(-y)\theta_q(-x/a)}{}_2\phi_2\left(\begin{matrix}0,q \\ aq/x,aq/y \end{matrix};q,-\frac{aq^2}{xy}\right)\right\} \\
\hphantom{{\rm i}q^{\frac{1}{8}}\mu(x,y;a)}{}
=
 {\rm e}^{\pi{\rm i}\alpha(u-v)}\left\{\frac{(a,q,xq,yq)_\infty}{\theta_q(-y)\theta_q(-x/a)}{}_1\phi_2\left(\begin{matrix}q \\ xq,yq \end{matrix};q,\frac{xyq^2}{a}\right) \right. \nonumber\\
\left.
\hphantom{{\rm i}q^{\frac{1}{8}}\mu(x,y;a) = {\rm e}^{\pi{\rm i}\alpha(u-v)}}{}
 +\frac{(a,q,x,y)_\infty}{\theta_q(-y)\theta_q(-x/a)}{}_3\phi_2\left(\begin{matrix}q/x,q/y,q \\ 0,0 \end{matrix};q,a\right)\right\}.
\end{gather*}
By Andrews' formula (\ref{eq:Andrews formula}), we obtain
\begin{gather*}
{\rm i}q^{\frac{1}{8}}\mu(x,y;a)
 =
 {\rm e}^{\pi{\rm i}\alpha(u-v)}\left\{-\frac{(aq)_\infty\theta_q(-yq/a)}{(q)_\infty\theta_q(-yq)}\Phi^{(1)}\left(\begin{matrix}a;0,0 \\ aq \end{matrix};q;\frac{xq}{a},\frac{yq}{a}\right) \right. \nonumber\\
\left.
\hphantom{{\rm i}q^{\frac{1}{8}}\mu(x,y;a) =
 {\rm e}^{\pi{\rm i}\alpha(u-v)}}{}
+\frac{(a,q,aq/x,aq/y;q)_\infty}{\theta_q(-y)\theta_q(-x/a)}{}_2\phi_2\left(\begin{matrix}0,q \\ aq/x,aq/y \end{matrix};q,-\frac{aq^2}{xy}\right)\right\} \\
\hphantom{{\rm i}q^{\frac{1}{8}}\mu(x,y;a)}{}
=
 {\rm e}^{\pi{\rm i}\alpha(u-v)}\left\{\frac{(a,q,xq,yq)_\infty}{\theta_q(-y)\theta_q(-x/a)}{}_1\phi_2\left(\begin{matrix}q \\ xq,yq \end{matrix};q,\frac{xyq^2}{a}\right) \right. \nonumber\\
\left.
\hphantom{{\rm i}q^{\frac{1}{8}}\mu(x,y;a) =
 {\rm e}^{\pi{\rm i}\alpha(u-v)}}{}
+\frac{(aq)_\infty\theta_q(-x)}{(q)_\infty\theta_q(-x/a)}\Phi^{(1)}\left(\begin{matrix}a;0,0 \\ aq \end{matrix};q;\frac{q}{x},\frac{q}{y}\right)\right\}.
\end{gather*}
Namely, $\mu(x,y;a)$ is regarded as a bilateral version of $q$-Appell hypergeometric functions
\[
\Phi^{(1)}\left(\begin{matrix}a;0,0 \\ aq \end{matrix};q;\frac{xq}{a},\frac{yq}{a}\right)
\qquad \text{or} \qquad
\Phi^{(1)}\left(\begin{matrix}a;0,0 \\ aq \end{matrix};q;\frac{q}{x},\frac{q}{y}\right).
\]

Further, $\mu(x,y;a)$ essentially satisfies the $q$-Appell difference system in the case of $b_1=b_2=0$, $c=aq$:
\begin{gather}
[(1-x-T_x)(1-aT_xT_y)]\Phi (x,y)=0, \nonumber\\
[(1-y-T_y)(1-aT_xT_y)]\Phi (x,y)=0, \nonumber\\
[x(1-T_y)-y(1-T_x)]\Phi (x,y)=0.\label{eq:q-Appell 2}
\end{gather}
\begin{thm}The function
\[
\nu(x,y;a):={\rm e}^{-\pi{\rm i}\alpha(u-v)}\frac{\theta_q(-ay)}{\theta_q(-y)}\mu(ax,ay;a)
\]
satisfies the multivariate $q$-difference equation \eqref{eq:q-Appell 2}.
More precisely, we have
\begin{gather}
\label{eq:more reduce}
(1-aT_xT_y)\nu(x,y;a)=0, \\
\label{eq:more reduce 2}
[x(1-T_y)-y(1-T_x)]\nu (x,y;a)=0.
\end{gather}
\end{thm}

\begin{proof}
To prove the first and second equations of (\ref{eq:q-Appell 2}), it is enough to show (\ref{eq:more reduce}).
By simple calculation, we have{\samepage
\begin{gather*}
aT_xT_y\nu(x,y;a)
=
 a\frac{\theta_q(-ayq)}{\theta_q(-yq)}\mu(axq,ayq;a) =
 a\frac{-a^{-1}y^{-1}}{-y^{-1}}\frac{\theta_q(-ay)}{\theta_q(-y)}\mu(x,y;a)
 =
 \nu(x,y;a).
\end{gather*}
Here, the second equality follows from $\tau$-periodicity (\ref{eq:main results4}).}

Another equation (\ref{eq:more reduce 2}) is proved as follows:
\begin{gather*}
(xT_y-yT_x)\nu(x,y;a)
 =
 x{\rm e}^{-\pi{\rm i}\alpha(u-v-\tau)}\frac{\theta_q(-ayq)}{\theta_q(-yq)}\mu(ax,ayq;a)\\
\hphantom{(xT_y-yT_x)\nu(x,y;a)=}{}
-y{\rm e}^{-\pi{\rm i}\alpha(u-v+\tau)}\frac{\theta_q(-ay)}{\theta_q(-y)}\mu(axq,ay;a)\\
\hphantom{(xT_y-yT_x)\nu(x,y;a)}{}
=
 {\rm e}^{-\pi{\rm i}\alpha(u-v)}\frac{\theta_q(-ay)}{\theta_q(-y)}\big\{ (-y\mu(ax,ay;a)+\sqrt{xy}\mu(ax,ay;a/q) ) \\
\hphantom{(xT_y-yT_x)\nu(x,y;a)=}{}
- (-x\mu(ax,ay;a)+\sqrt{xy}\mu(ax,ay;a/q) )\big\}\\
\hphantom{(xT_y-yT_x)\nu(x,y;a)}{}
=(x-y)\nu(x,y;a).
\end{gather*}
The first and second equalities in the above follow from the forward shift $(\ref{eq:main results2})$ and $(\ref{eq:v+tau})$, respectively.
\end{proof}

In particular, Zwegers' $\mu$-function $\nu(x,y;q)=-{\rm e}^{-\pi{\rm i}(u+v)}\mu(x,y)$ also is a solution of the case of $a=q$ in the two-variate $q$-difference system (\ref{eq:q-Appell 2}):
\begin{gather}
[1-qT_xT_y]\nu (x,y;q)=0, \nonumber\\
[x(1-T_y)-y(1-T_x)]\nu (x,y;q)=0.\label{eq:q-Appell 3}
\end{gather}
Based on this fact, it would be desirable to study other generalizations and their global analysis of the $\mu$-function $\mu(u,v)$ from the view of analysis of the $q$-Appell difference system (\ref{eq:q-Appell 3}).

\subsection*{Acknowledgments}
We are grateful to Professor Yasuhiko Yamada (Kobe University) for his helpful advice on our paper.
We also wish to thank Professor Yosuke Ohyama (Tokushima University) for his valuable suggestions on $q$-special functions, including in the unpublished proof of Lemma~\ref{lem:lemma2}~\cite{O2}.
We are also indebted to Professor Kazuhiro Hikami (Kyushu University) for his information for~\cite{GW} and quantum invariants.
Professor Toshiki Matsusaka (Kyushu University) also provides information on~\cite{GW} and his note~\cite{M}.
Some pieces of information on the $q$-Appell hypergeometric function~$\Phi ^{(1)}$ and its $q$-difference equations are provided by Dr.\ T.~Nobukawa.
Finally, we thank the referees for their helpful comments about mock and indefinite theta functions.
This work was supported by JSPS KAKENHI Grant Number 21K13808.

\pdfbookmark[1]{References}{ref}
\LastPageEnding


\begin{thebibliography}{99}
\footnotesize\itemsep=0pt

\bibitem{AB}
Andrews G.E., Berndt B.C., Ramanujan's lost notebook. {P}art~{V}, \href{https://doi.org/10.1007/978-3-319-77834-1}{Springer},
 Cham, 2018.

\bibitem{AH}
Andrews G.E., Hickerson D., Ramanujan's ``lost'' notebook.~{VII}. {T}he sixth
 order mock theta functions, \href{https://doi.org/10.1016/0001-8708(91)90083-J}{\textit{Adv. Math.}} \textbf{89} (1991), 60--105.

\bibitem{BW}
Beals R., Wong R., Special functions. A graduate text, \textit{Cambridge Stud.
 Adv. Math.}, Vol.~126, \href{https://doi.org/10.1017/CBO9780511762543}{Cambridge University Press}, Cambridge, 2010.

\bibitem{B}
Bradley-Thrush J.G., Properties of the {A}ppell--{L}erch function~({I}),
 \href{https://doi.org/10.1007/s11139-021-00445-4}{\textit{Ramanujan~J.}} \textbf{57} (2022), 291--367.

\bibitem{BFOR}
Bringmann K., Folsom A., Ono K., Rolen L., Harmonic {M}aass forms and mock
 modular forms: theory and applications, \textit{Amer. Math. Soc. Colloq.
 Publ.}, Vol.~64, \href{https://doi.org/10.1090/coll/064}{Amer. Math. Soc.}, Providence, RI, 2017.

\bibitem{C}
Choi Y.S., The basic bilateral hypergeometric series and the mock theta
 functions, \href{https://doi.org/10.1007/s11139-010-9269-7}{\textit{Ramanujan~J.}} \textbf{24} (2011), 345--386.

\bibitem{GW}
Garoufalidis S., Wheeler C., Modular $q$-holonomic modules,
 \href{https://arxiv.org/abs/2203.17029}{arXiv:2203.17029}.

\bibitem{GR}
Gasper G., Rahman M., Basic hypergeometric series, 2nd ed., \textit{Encyclopedia Math.
 Appl.}, Vol.~96, \href{https://doi.org/10.1017/CBO9780511526251}{Cambridge University Press}, Cambridge, 2004.

\bibitem{G}
Gauss C.F., Summatio quarumdam serierum singularium, \textit{Comm. Soc. Reg.
 Sci. Gottingensis Rec.} \textbf{1} (1811), 1--40.

\bibitem{GM}
Gordon B., McIntosh R.J., A survey of classical mock theta functions, in
 Partitions, {$q$}-Series, and Modular Forms, \textit{Dev. Math.}, Vol.~23,
 \href{https://doi.org/10.1007/978-1-4614-0028-8_9}{Springer}, New York, 2012, 95--144.

\bibitem{H}
Hickerson D., A proof of the mock theta conjectures, \href{https://doi.org/10.1007/BF01394279}{\textit{Invent. Math.}}
 \textbf{94} (1988), 639--660.

\bibitem{K}
Kang S.Y., Mock {J}acobi forms in basic hypergeometric series, \href{https://doi.org/10.1112/S0010437X09004060}{\textit{Compos.
 Math.}} \textbf{145} (2009), 553--565, \href{https://arxiv.org/abs/0806.1878}{arXiv:0806.1878}.

\bibitem{KLS}
Koekoek R., Lesky P.A., Swarttouw R.F., Hypergeometric orthogonal polynomials
 and their {$q$}-analogues, \textit{Springer Monogr. Math.}, \href{https://doi.org/10.1007/978-3-642-05014-5}{Springer}, Berlin, 2010.

\bibitem{Ko}
Koelink H.T., Hansen--{L}ommel orthogonality relations for {J}ackson's
 {$q$}-{B}essel functions, \href{https://doi.org/10.1006/jmaa.1993.1181}{\textit{J.~Math. Anal. Appl.}} \textbf{175} (1993),
 425--437.

\bibitem{M}
Matsuzaka T., {P}rivate communication, 2022.

\bibitem{O1}
Ohyama Y., A unified approach to $q$-special functions of the {L}aplace type,
 \href{https://arxiv.org/abs/1103.5232}{arXiv:1103.5232}.

\bibitem{O2}
Ohyama Y., {P}rivate communication, 2022.

\bibitem{RSZ}
Ramis J.P., Sauloy J., Zhang C., Local analytic classification of
 {$q$}-difference equations, \textit{Ast\'erisque} \textbf{355} (2013),
 vi+151~pages, \href{https://arxiv.org/abs/0903.0853}{arXiv:0903.0853}.

\bibitem{S}
Suslov S.K., Some orthogonal very-well-poised {$_8\phi_7$}-functions that
 generalize {A}skey--{W}ilson polynomials, \href{https://doi.org/10.1023/A:1011439924912}{\textit{Ramanujan~J.}} \textbf{5}
 (2001), 183--218, \href{https://arxiv.org/abs/math.CA/9707213}{arXiv:math.CA/9707213}.

\bibitem{W}
Weil A., Elliptic functions according to {E}isenstein and {K}ronecker, \textit{Classics
 Math.}, Springer, Berlin, 1999.

\bibitem{R}
Westerholt-Raum M., {${\rm H}$}-harmonic {M}aa{\ss}--{J}acobi forms of
 degree~1, \href{https://doi.org/10.1186/s40687-015-0032-y}{\textit{Res. Math. Sci.}} \textbf{2} (2015), art.~12, 34~pages.

\bibitem{Zh}
Zhang C., Une sommation discr\`ete pour des \'equations aux {$q$}-diff\'erences
 lin\'eaires et \`a coefficients analytiques: th\'eorie g\'en\'erale et
 exemples, in Differential Equations and the {S}tokes Phenomenon, \href{https://doi.org/10.1142/9789812776549_0012}{World Sci.
 Publ.}, River Edge, NJ, 2002, 309--329.

\bibitem{Zw1}
Zwegers S.P., Mock theta functions, Ph.D.~Thesis, {U}niversiteit Utrecht, 2002,
 available at \url{https://dspace.library.uu.nl/handle/1874/878}.

\end{thebibliography}
\end{document}